\documentclass[12pt]{amsart}
\usepackage{amssymb,amscd}
\usepackage{verbatim}

\usepackage{amsmath,amssymb,graphicx,mathrsfs}   
\usepackage[colorlinks=true,allcolors = blue]{hyperref} 

\usepackage{tikz}
\usetikzlibrary{matrix}

\usepackage[all]{xy}

\let\frak\mathfrak

\def\>{\relax\ifmmode\mskip.666667\thinmuskip\relax\else\kern.111111em\fi}
\def\<{\relax\ifmmode\mskip-.333333\thinmuskip\relax\else\kern-.0555556em\fi}
\def\vsk#1>{\vskip#1\baselineskip}
\def\vv#1>{\vadjust{\vsk#1>}\ignorespaces}
\def\vvn#1>{\vadjust{\nobreak\vsk#1>\nobreak}\ignorespaces}

  \let\ssize\scriptstyle
\let\sssize\scriptscriptstyle

\let\Medskip\medskip
\def\medskip{\par\Medskip}
\let\Bigskip\bigskip
\def\bigskip{\par\Bigskip}

\let\Maketitle\maketitle
\def\maketitle{\Maketitle\thispagestyle{empty}\let\maketitle\empty}

\newtheorem{thm}{Theorem}[section]
\newtheorem{cor}[thm]{Corollary}
\newtheorem{lem}[thm]{Lemma}
\newtheorem{prop}[thm]{Proposition}

\theoremstyle{definition}                                  
\numberwithin{equation}{section}

\theoremstyle{definition}
\newtheorem*{rem}{Remark}

\let\mc\mathcal
\let\nc\newcommand

\let\al\alpha

\let\dl\delta

\let\eps\varepsilon

\let\ka\kappa
\let\la\lambda

\let\phi\varphi
\let\si\sigma

\let\om\omega
\let\Om\Omega

\let\der\partial

\let\ox\otimes

\let\geq\geqslant

\let\leq\leqslant

\let\on\operatorname
\let\bi\bibitem
\let\bs\boldsymbol

\def\C{{\mathbb C}}
\def\Z{{\mathbb Z}}

\def\F{{\mathbb F}}   

\def\+#1{^{\{#1\}}}

\def\id{{\on{id}}}

\def\beq{\begin{equation}}
\def\eeq{\end{equation}}
\def\be{\begin{equation*}}
\def\ee{\end{equation*}}

\nc{\bea}{\begin{eqnarray*}}
\nc{\eea}{\end{eqnarray*}}
\nc{\bean}{\begin{eqnarray}}
\nc{\eean}{\end{eqnarray}}

\def\g{{\mathfrak g}}

\let\ga\gamma

\nc{\Il}{{\mc I_{\bs\la}}}
\nc{\bla}{{\bs\la}}
\nc{\Fla}{\F_\bla}
\nc{\tfl}{{T^*\Fla}}
\nc{\GL}{{GL_n(\C)}}
\nc{\GLC}{{GL_n(\C)\times\C^*}}

\let\sd s 

\def\ddk_#1{\kk_{#1}\<\>\frac\der{\der\<\>\kk_{#1}}}

\def\bul{\mathbin{\raise.2ex\hbox{$\sssize\bullet$}}}
\def\intt{\mathchoice
{\mathop{\raise.2ex\rlap{$\,\,\ssize\backslash$}{\intop}}\nolimits}
{\mathop{\raise.3ex\rlap{$\,\sssize\backslash$}{\intop}}\nolimits}
{\mathop{\raise.1ex\rlap{$\sssize\>\backslash$}{\intop}}\nolimits}
{\mathop{\rlap{$\sssize\<\>\backslash$}{\intop}}\nolimits}}

\let\kk q 
\let\cc c

\let\Ko K

\def\GZ/{Gelfand-Zetlin}
\def\KZ/{{\slshape KZ\/}}
\def\qKZ/{{\slshape qKZ\/}}
\def\XXX/{{\slshape XXX\/}}

\nc{\A}{{\mc C}}

\def\sll{{\frak{sl}}}

\def\Q{{\mathbb Q}}

\nc{\hsl}{\widehat{{\frak{sl}_2}}}

\nc{\BC}{{ \mathbb C}}
\nc{\lra}{\longrightarrow}
\nc{\CO}{{\mathcal{O}}}
\nc{\BZ}{{ \mathbb Z}}
\nc{\hfn}{\hat{\frak{n}}}
\nc\Zs{{\Z/p^s\Z}}
\nc\Zo{{\Zs[z]^0}}
\nc\gr{{\on{gr}}}

\nc\fD{{\frak D}}

\begin{document}

\hrule width0pt
\vsk->

\title[Notes on solutions of KZ equations modulo $p^s$]
{Notes on solutions of KZ equations
\\
 modulo $p^s$
 and $p$-adic limit $s\to\infty$}

\author[Alexander Varchenko]
{ Alexander Varchenko$^{\star}$}

\maketitle

\begin{center}
{with an appendix by Steven Sperber$^{\circ}$ and Alexander Varchenko}
\end{center}

\bigskip
\medskip

\begin{center}
{\it $^\star\<$Department of Mathematics, University
of North Carolina at Chapel Hill\\ Chapel Hill, NC 27599-3250, USA\/}

\vsk.5>
{\it $^\star\<$Faculty of Mathematics and Mechanics, Lomonosov Moscow State
University\\ Leninskiye Gory 1, 119991 Moscow GSP-1, Russia\/}

\vsk.5>
 {\it $^\star\<$Moscow Center of Fundamental and Applied Mathematics
\\ Leninskiye Gory 1, 119991 Moscow GSP-1, Russia\/}

\vsk.5>
 {\it $^\circ\<$School of Mathematics, University of Minnesota
 \\    	127 Vincent Hall,  Minneapolis, MN 55455, USA \/}

\end{center}

\vsk>
{\leftskip3pc \rightskip\leftskip \parindent0pt \Small
{\it Key words\/}:  KZ equations,  reduction modulo $p^s$,
$p^s$-hypergeometric solutions, $p$-adic limit,
Frobenius transformations, unit roots

\vsk.6>
{\it 2010 Mathematics Subject Classification\/}: 13A35 (11G25, 14G10, 33C60, 32G20) 
\par}

{\let\thefootnote\relax
\footnotetext{\vsk-.8>\noindent
$^\star\<$
{\it E\>-mail}:
anv@email.unc.edu\,, supported in part by NSF grant DMS-1954266
\\
$^\circ\<$ {\it E\>-mail}: sperber@umn.edu}}

\newpage

\begin{abstract}
We consider the  differential KZ equations over $\mathbb C$  in the case, when the hypergeometric
solutions are one-dimensional hyperelliptic integrals of genus $g$. 
In this case the space of solutions of the differential KZ equations is a $2g$-dimensional complex vector space.

We also consider the same differential
equations modulo $p^s$, where $p$ is an odd prime number and $s$ is
 a positive integer, and over the field  $\mathbb Q_p$ of $p$-adic
numbers.

We describe a construction of polynomial solutions of the
differential  KZ equations modulo $p^s$.
These polynomial solutions have integer coefficients and
are $p^s$-analogs of the hyperelliptic integrals.
We call them the $p^s$-hypergeometric solutions.  
We consider the space $\mathcal M_{p^s}$ of all $p^s$-hypergeometric solutions,  which is a module over the ring
of polynomial quasi-constants modulo $p^s$. We study basic properties of $\mathcal M_{p^s}$,
in particular its natural filtration, and the dependence of $\mathcal M_{p^s}$ on $s$.

We show that the $p$-adic limit of $\mathcal M_{p^s}$ as $s\to\infty$ gives us a
 $g$-dimensional vector space of solutions of the differential KZ equations over
the field $\mathbb Q_p$. The solutions over $\mathbb Q_p$  are power series at a
certain asymptotic zone of the KZ equations.

In the appendix written jointly with Steven Sperber we consider 
all asymptotic zones of the KZ equations in the special case $g=1$ of elliptic integrals.
It turns out that  in this case the $p$-adic limit of  
$\mathcal M_{p^s}$ as $s\to \infty$ gives us a one-dimensional space of solutions over $\mathbb Q_p$ at every
asymptotic zone. We apply Dwork's theory of the classical hypergeometric function over $\mathbb Q_p$ and show that
our  germs of solutions over $\mathbb Q_p$ defined at different asymptotic zones analytically continue into
a single global invariant line subbundle of the associated KZ connection. Notice that the corresponding KZ connection
over $\mathbb C$  does not have proper nontrivial invariant subbundles, and therefore our invariant line subbundle is
a new feature of the KZ equations over $\mathbb Q_p$.

Also in the appendix we follow Dwork and describe the Frobenius transformations of solutions of the KZ equations for $g=1$.
Using these Frobenius transformations we recover the unit roots of the zeta functions of the elliptic curves defined by
the affine equations 
$y^2= \beta \,x(x-1)(x-\alpha)$ over the finite field 
$\mathbb F_p$. Here $\alpha,\beta\in\mathbb F_p^\times, \alpha \ne 1$.
 Notice that the same elliptic curves considered over $\C$ are used
to construct the complex holomorphic solutions of the KZ equations for $g=1$.

\end{abstract}

{\small\tableofcontents\par}

\setcounter{footnote}{0}
\renewcommand{\thefootnote}{\arabic{footnote}}

\section{Introduction}

\noindent{\bf 1.1.}
The KZ equations were introduced  in  \cite{KZ} as  the differential equations satisfied by 
conformal blocks on sphere in the Wess-Zumino-Witten model of conformal field theory.
The solutions of the KZ equations in the form of multidimensional hypergeometric integrals
were constructed more than  30 years ago,
see \cite{SV1}. 
The KZ equations and the hypergeometric solutions 
are related to many subjects in algebra, representation theory,
theory of integrable systems, enumerative geometry.

\vsk.2>

 The polynomial
solutions of the KZ equations over the finite field $\F_p$  
of  a prime number $p$ of elements were constructed relatively  recently 
 in \cite{SV2}, see also \cite{V4}-\cite{V8}, \cite{RV1, RV2}.
These  solutions  were called the { $\F_p$-hypergeometric solutions}.
  The general problem is to understand relations between the hypergeometric solutions 
of the KZ equations over $\C$ and the $\F_p$-hypergeometric solutions and observe how the
remarkable properties of hypergeometric solutions are reflected in the properties
of the $\F_p$-hypergeometric solutions. 
For example, the $\F_p$-hypergeometric solutions inherit some determinant properties of
the hypergeometric solutions and some Selberg integral properties, see \cite{V8, RV1, RV2}.

\vsk.2>
This program is in the first stages, where we consider  essential  examples and 
study the corresponding $\F_p$-hypergeometric solutions by direct methods.

\vsk.2>

In this paper we consider the  differential KZ equations over $\C$  in the case, when the hypergeometric
solutions are one-dimensional hyperelliptic integrals of genus $g$. 
In this case the space of solutions of the differential KZ equations is a $2g$-dimensional complex vector space.
We also consider the same differential
equations modulo $p^s$, where $p$ is an odd prime number and $s$ is
 a positive integer, and over the field  $\Q_p$ of $p$-adic
numbers.

\vsk.2>
We give a construction of polynomial solutions of the 
differential KZ equations modulo 
$p^s$ for positive integers $s$.   We call such solutions the {\it $p^s$-hypergeometric solutions.}
This construction is a straightforward modification of the construction in
\cite{SV2} of polynomial solutions  modulo $p$. 

\vsk.2>
In this paper we consider the space $\mc M_{p^s}$ of all $p^s$-hypergeometric solutions,  which is a module over the ring
of polynomial quasi-constants modulo $p^s$. We study basic properties 
 of $\mc M_{p^s}$, in particular its natural filtration, and dependence of $\mc M_{p^s}$ on $s$.
 
 \vsk.2>
 We show that the $p$-adic limit of $\mc M_{p^s}$ as $s\to\infty$ gives us a
 $g$-dimensional vector space of solutions of the differential KZ equations over
the field $\Q_p$. The solutions over $\Q_p$  are power series at a
certain asymptotic zone of the KZ equations. This is the main result of the paper, see
Lemma \ref{lem padind} and Theorem \ref{thm last}.

\vsk.4>
\noindent{\bf 1.2.}
In the appendix written jointly with Steven Sperber we consider 
all six asymptotic zones of the KZ equations in the special case $g=1$ of elliptic integrals.
It turns out that  in this case the $p$-adic limit of  
$\mc M_{p^s}$ as $s\to \infty$ gives us a one-dimensional space of solutions over $\Q_p$ at every
asymptotic zone. We apply Dwork's theory
of the classical hypergeometric function over $\Q_p$ and show that
our  germs of solutions over $\Q_p$ defined at different asymptotic zones analytically continue into
a single global invariant line subbundle of  the associated KZ connection. Notice that the corresponding KZ connection
over $\C$  does not have proper nontrivial invariant subbundles, and therefore our invariant line subbundle is 
a new feature of the KZ equations over $\Q_p$. 

\vsk.2>
Following Dwork we show that our line subbundle is spanned
at any point of the base by the germs of all solutions of the KZ equations bounded in their discs of convergence.
This statement gives a definition of the line subbundle independent of asymptotic zones and analytic continuation.

\vsk.2>
Also in the appendix we follow Dwork and describe the Frobenius transformations of solutions of the KZ equations for $g=1$.
Using these Frobenius transformations we recover the unit roots of the zeta functions of the elliptic curves defined by
the affine equations 
$y^2= \beta \,x(x-1)(x-\al)$ over the finite field $\F_p$. Here $\al,\beta\in\F_p^\times, \al \ne 1$.
 Notice that the same elliptic curves considered over $\C$ are used
to construct the complex holomorphic solutions of the KZ equations for $g=1$.

In the end of Section \ref{sec A10} we argue that the KZ equations for $g=1$ contain
more arithmetic information than the associated hypergeometric differential equation \eqref{HE}  for the hypergeometric function 
$I(z)$ in \eqref{1}, studied in \cite{Dw}.

\vsk.4>
\noindent{\bf 1.3.}
Our   $p$-adic limit of $\mc M_{p^s}$ as $s\to\infty$ 
is similar to the $p$-adic limit in the following classical example, see \cite{Ig, Ma,Cl, BV1}.
Consider the elliptic integral
\bean
\label{1}
I(z) = \frac 1\pi\,\int_1^\infty \frac{dx}{\sqrt{x(x-1)(x-z)}}=\sum_{k = 0}^\infty\binom{-1/2}{k}^2z^k\,.
\eean
It satisfies the hypergeometric differential equation 
\bean
\label{HE}
z(1-z) I'' +(1-2z)I'-(1/4)I=0.
\eean
The coefficients of the power series $I(z)$ are $p$-adic integers and the power series 
$I(z)$ converges $p$-adically for 
$|z|_p<1$, where $|z|_p$ is the $p$-adic norm of $z\in\Q_p$. 
One may show that for any positive integer $s$ 
the polynomial 
\bean
\label{2}
I_{(p^s-1)/2}(z) = \sum_{k = 0}^{(p^s-1)/2} \binom{(p^s-1)/2}{k}^2z^k\,
\eean
 is a solution of  the differential equation \eqref{HE} modulo $p^s$. 
Thus we get a sequence 
\linebreak
$(I_{(p^s-1)/2}(z))_{s=1}^\infty$ of polynomials with integer coefficients,
each of which is a solution of the differential equation \eqref{HE} modulo
$p^s$, and the $p$-adic limit of the sequence, as $s$ tends to $\infty$,
is the $p$-adic power series solution $I(z)$ of the differential equation \eqref{HE}.

\vsk.2>
The $p^s$-hypergeometric solutions of our differential KZ equations are analogs of the polynomials 
$I_{(p^s-1)/2}(z)$ with an  analogous $p$-adic limit.
 The difference is that the construction of the 
 $p^s$-hypergeometric solutions does not indicate the
 analogous  $p$-adic limiting  solutions $I(z)$, 
and the analogous limiting $p$-adic power series solutions $I(z)$
can be discovered only after rewriting
the $p^s$-hypergeometric solutions in a suitable asymptotic zone of the differential KZ equations.

\vsk.2>

In the simplest example of our differential KZ equations,
the $p$-adic solution  is the 3-vector
\bean
\label{3}
I(u_1, u_2) 
&=&
 u_1^{-3/2} \sum_{k=0}^\infty \binom{-1/2}{k+1}\binom{-3/2}{k}
\Big(\frac{k+1}{-1/2-k},1,\frac{-1/2}{-1/2-k}\Big) u_2^k\,,
\eean
while the sequence $(I_{(p^s-3)/2}(u_1,u_2))_{s=1}^\infty$ 
of the $p^s$-hypergeometric solutions modulo $p^s$ of the same equations
is given by the formula

\bean
\label{4}
I_{(p^s-3)/2}(u_1,u_2) 
&=&
 u_1^{(p^s-3)/2} 
\sum_{k=0}^{(p^s-3)/2} \binom{(p^s-1)/2}{k+1}\binom{(p^s-3)/2}{k}
\\
\notag
&\times&
\Big(\frac {k+1}{(p^s-1)/2-k},1,\frac {(p^s-1)/2}{(p^s-1)/2-k}\Big) u_2^k\,,
\eean
 see Section \ref {sec Exn3}.  
 
 \vsk.2>
 The sum $\sum_{k=0}^{(p^s-3)/2}$ in \eqref{4}
 is the truncation of the sum $\sum_{k=0}^{\infty}$ in \eqref{3}, similar to what happens in \eqref{1} and \eqref{2}.
 A new feature appears when we compare the prefactor $u_1^{-3/2}$ and the sequence of prefactors 
 $\big(u_1^{(p^s-3)/2}\big)_{s=1}^\infty$.   As $s\to\infty$ the sequence of prefactors
  $\big(u_1^{(p^s-3)/2}\big)_{s=1}^\infty$  tends $p$-adically to the prefactor $u^{-3/2}$ 
multiplied by a Teichmuller constant on a suitable domain in $\Z_p$, 
where $\Z_p$ is the ring of $p$-adic integers, 
see Section \ref{sec n=3 c} and Theorem \ref{thm 10.3}.

\vsk.4>
\noindent{\bf 1.4.}
The paper is organized as follows. In Section \ref{sec DE} we define our system of KZ equations. In Section \ref{sec2}
 we describe its complex solutions as hyperelliptic integrals. In Section \ref{sec4} we describe the
 $p^s$-hypergeometric solutions of our KZ equations modulo 
 $p^s$ and define the filtered  module $\mc M_{p^s}$ of all $p^s$-hypergeometric 
solutions.   In Section \ref{sec5} we prove the independence of the module $\mc M_{p^s}$ from some arithmetic 
data involved in its definition. 
In Section \ref{sec6} we discuss the properties of the operator $\mc M_{p^s}\to \mc M_{p^s}$ of multiplication by $p$.
In Section \ref{sec7} we calculate the coefficients of the Taylor expansion of the $p^s$-hypergeometric solutions.
In Section \ref{sec8} we relate the operator $\mc M_{p^s}\to \mc M_{p^s}$
 of multiplication by $p$  and the Cartier-Manin 
matrix associated with the hyperelliptic curve defined by the affine equation $y^2=(x-z_1)\cdots(x-z_n)$.
In Section \ref{sec9} we consider one of the asymptotic zones of our KZ equations.
Using the coordinates in that asymptotic zone we describe the $p$-adic limit of the $p^s$-hypergeometric solutions in Section
\ref{sec10}. In Appendix \ref{appendix} we apply Dwork's theory in \cite{Dw} to the case $g=1$.
In Section \ref{sec A12} we discuss open problems related to the case of an arbitrary $g$.

\smallskip
The author thanks Masha Vlasenko for numerous clarifying remarks on  basics of the
$p$-adic theory of geometric differential equations and comments on  drafts of this paper.
The author thanks Steven Sperber for collaboration on this project.
The author thanks Pavel Etingof and Vadim Schechtman for useful discussions.

\section{KZ equations} 
\label{sec DE} 

Let $\g$ be a simple Lie algebra with an invariant scalar product.
The {\it Casimir element}  is 
\bea
\Om = {\sum}_i \,h_i\ox h_i \ \ \in \ \g \ox \g,
\eea
where $(h_i)\subset\g$ is an orthonormal basis.
Let  $V=\otimes_{i=1}^n V_i$ be 
a tensor product of $\g$-modules, $\ka\in\C^\times$ a nonzero number.
The {\it differential KZ equations} is the system of differential equations on a $V$-valued function $I(z_1,\dots,z_n)$,
\bea
\frac{\der I}{\der z_i}\ =\ \frac 1\ka\,{\sum}_{j\ne i}\, \frac{\Om_{i,j}}{z_i-z_j} I, \qquad i=1,\dots,n,
\eea
where $\Om_{i,j}:V\to V$ is the Casimir operator acting in the $i$th and $j$th tensor factors,
see \cite{KZ, EFK}.

\vsk.2>

This system is a system of Fuchsian first order
 linear differential equations. 
  The equations are defined on the complement in $\C^n$ to the union of all diagonal hyperplanes.
 
\vsk.2>

The object of our discussion is the following particular case. 

\vsk.2>
Let  $p$ {\it be an odd prime number,
 $n=2g+1$  an odd positive integer, $p>n\geq 2$.}
We  study the system of equations
for a  column vector  $I(z)=(I_1(z)$, \dots, $I_{n}(z))$\,:
\bean
\label{KZ}
\phantom{aaa}
 \frac{\partial I}{\partial z_i} \ = \
   {\frac 12} \sum_{j \ne i}
   \frac{\Omega_{ij}}{z_i - z_j}  I ,
\quad i = 1, \dots , n,
\qquad
I_1(z)+\dots+I_{n}(z)=0,
\eean
where $z=(z_1,\dots,z_n)$,
the $n\times n$-matrices $\Om_{ij}$ have the form:
\bean
\label{Om_ij_reduced}
 \Omega_{ij} \ = \ \begin{pmatrix}
             & \vdots^i &  & \vdots^j &  \\
        {\scriptstyle i} \cdots & {-1} & \cdots &
            1   & \cdots \\
                   & \vdots &  & \vdots &   \\
        {\scriptstyle j} \cdots & 1 & \cdots & -1&
                 \cdots \\
                   & \vdots &  & \vdots &
                   \end{pmatrix} ,
\eean                    
and all other entries are zero.
 This  joint system of {\it differential and 
algebraic equations} will be called the {\it system of KZ  equations} in this paper. 

\vsk.2>
System  \eqref{KZ} is the system of the  differential KZ equations with parameter $\ka=2$ associated with the Lie algebra $\sll_2$ and the subspace of singular vectors of weight $2g-1$ of the tensor power 
$(\C^2)^{\ox {(2g+1)}}$ of two-dimensional irreducible $\sll_2$-modules, up to a gauge transformation, see 
this example in  \cite[Section 1.1]{V3}.

\vsk.2>
We consider system \eqref{KZ} over the field $\C$. We
 also consider the same system of equations modulo $p^s$
 and over the field  $\Q_p$ of $p$-adic  numbers.

\section{Complex solutions}
\label{sec2}

Consider the {\it master function}
\bean
\label{mast f}
\Phi(x,z) = \prod_{a=1}^{n}(x-z_a)^{-1/2}
\eean
and  the column ${n}$-vector  of hyperelliptic  integrals
\bean
\label{Iga}
I^{(\ga)} (z)=(I_1(z),\dots,I_n(z)),
\qquad
I_j=\int   \frac{\Phi(x,z)}{x-z_j}\,dx\,.
\eean
The integrals $I_j$, are over an element $\ga$ of the first homology group
 of the algebraic curve with affine equation
\bea
y^2 = (x-z_1)\dots (x-z_{n})\,.
\eea
Starting from such $\ga$,  chosen for given values
$\{z_1,\dots,z_{n}\}$, the vector $I^{(\ga)}(z)$ can 
be analytically continued as a multivalued holomorphic function of $z$ to the complement in $\C^n$ of the union of the
diagonal  hyperplanes $z_i=z_j$, $i\ne j$.

\begin{thm}
\label{thm1.1}

 The vector $I^{(\ga)}(z)$ is a solution of system  \eqref{KZ}.
\end{thm}

Theorem \ref{thm1.1} is a classical statement. 
Much more general algebraic and differential equations satisfied by analogous multidimensional hypergeometric integrals were considered in \cite{SV1}.  Theorem \ref{thm1.1} is discussed as an example in  \cite[Section 1.1]{V3}.

\begin{proof}  
The theorem follows from  Stokes' theorem and the two identities:
\bean
\label{i1}
-\frac 12\,
\Big(\frac {\Phi(x,z)}{x-z_1} + \dots + \frac {\Phi(x,z)}{x-z_n}\Big)\,  =\, \frac{\der\Phi}{\der x}(x,z)\,,
\eean
\bean
\label{i2}
\Big(\frac{\der }{\der z_i}-\frac12
\sum_{j\ne i} \frac {\Omega_{i,j}}{z_i-z_j} \Big)\Big(\frac {\Phi(x,z)}{x-z_1}, \dots, \frac {\Phi(x,z)}{x-z_n}\Big) 
= \frac{\der \Psi^i}{\der x} (x,z),
\eean
where  $\Psi^i(x,z)$ is the column $n$-vector   $(0,\dots,0,-\frac{\Phi(x,z)}{x-z_i},0,\dots,0)$ with 
the nonzero element at the $i$-th place. 
\end{proof}

\begin{thm} [{\cite[Formula (1.3)]{V1}}]
\label{thm dim}

All solutions of system \eqref{KZ} have this form. 
Namely, the complex vector space of solutions of the form \eqref{Iga} is $n-1$-dimensional.

\end{thm}

This theorem follows from the determinant formula for multidimensional hypergeometric integrals  in    \cite{V1}, in particular,
from \cite[Formula (1.3)]{V1}.

\section{Solutions modulo $p^s$}
\label{sec4}

\subsection{Leading terms}
For a ring $R$ denote  $R[z]=R[z_1,\dots,z_n]$. For a positive integer $t$ denote $R[z^{p^t}]=R[z_1^{p^t}, \dots,z_n^{p^t}]$.

\vsk.2>

Consider the lexicographical ordering of monomials
$z_1^{d_1}\dots z_n^{d_n}$,
 so we have $z_1>\dots > z_n$ and  so on.
For a nonzero polynomial
$f(z)=\sum_{d_1,\dots,d_n} a_{d_1,\dots,d_n} z_1^{d_1}\dots z_n^{d_n}$\,
let $f_{\frak l}(z)$ be the nonzero summand
$a_{d_1,\dots,d_n} z_1^{d_1}\dots z_n^{d_n}$ with the largest monomial
$z_1^{d_1}\dots z_n^{d_n}$.
 We call $f_{\frak l}(z)$ 
 the {\it leading term} of $f(z)$, the coefficient
$a_{d_1,\dots,d_n}$ -- the {\it leading coefficient}, the 
monomial
$z_1^{d_1}\dots z_n^{d_n}$ -- the {\it leading monomial}.

\vsk.2>
Let $s$ be a positive integer.
An element  $a\in\Zs$ has a unique presentation $a=a_0+a_1p+\dots+a_{s-1}p^{s-1}$,
where
$a_i\in \{0,\dots,p-1\}$. An element $a$ is invertible  if and only if $a_0\ne 0$.

\vsk.2>
Denote  $\F_p=\Z/p\Z$.

\vsk.2>
Let $\pi_{s}$ denote the homomorphisms 
$\Z \to \Zs$, $\Z[z] \to (\Zs)[z]$,  $\Z[z]^n \to (\Zs)[z]^n$
and for $t<s$ let $\pi_{s,t}$ denote the homomorphisms
$\Zs \to \Z/p^t\Z$, $(\Zs)[z] \to (\Z/p^t\Z)[z]$,  $(\Zs)[z]^n \to (\Z/p^t\Z)[z]^n$.

\vsk.2>

\subsection{Quasi-constants}

We say that a polynomial $f(z)\in\Z[z]$ is a {\it quasi-constant modulo} $p^s$ if
$\frac{\der f}{\der z_i}\in p^s\Z[z]$ for $i=1,\dots,n$. 
The quasi-constants modulo $p^s$ form a subring
of $\Z[z]$ denoted by $\Z[z]_{p^s}$.  For example $(z_1+z_2)^{p^s}\in\Z[z]_{p^s}$.

\begin{lem}
\label{lem quasi-c}

As a $\Z$-module  the ring $\Z[z]_{p^s}$ is spanned by the monomials
$p^{s-t} z_1^{d_1}\dots z_n^{d_n}$, where $t$ is the maximal
integer such that $t\leq s$ and $p^{t}$ divides every $d_1, \dots, d_n$.  

\end{lem}

For example, $z_1^{lp^s}$ and $p^{s-1} z_1^{p}z_2^{p^2}$ 
are such monomials.

\begin{proof}
Let $f(z)=\sum_d c_dz_1^{d_1}\dots z_n^{d_n} \in  \Z[z]_{p^s}$. We show that each summand
$c_dz_1^{d_1}\dots z_n^{d_n}$ is a multiple of a monomial of Lemma \ref{lem quasi-c}. Indeed, let
$c_{d^0}z_1^{d_1^0}\dots z_n^{d_n^0}$ be the leading term of $f(z)$.
 Then all first partial derivatives of it must lie in $p^s\Z[z]$
Hence $c_{d^0}\in p^{s-t} \Z$, where $t$ is the maximal integer such that $t\leq s$ and $p^{t}$
 divides every $d_1^0, \dots, d_n^0$. Subtracting
the leading term from $f(z)$ and repeating the reasoning we prove the lemma.
\end{proof}

\begin{lem}
\label{lem qc2} Let $f(z)$ be a quasi-constant modulo $p^s$ and  $t\in\Z_{\geq 0}$. 
Then $p^tf(z)$ is a quasi-constant modulo $p^{r}$ for any $1\leq r\leq s+t$.
\qed
\end{lem}

The rings of quasi-constants form a decreasing filtration,
$\Z[z]_{p}\supset \Z[z]_{p^2} \supset \dots\,.$

\subsection{Solutions of system \eqref{KZ} modulo $p^s$}

We say that a column $n$-vector $I(z) \in\Z[z]^n$ of polynomials with integer coefficients
is a {\it solution of system \eqref{KZ} modulo} $p^s$, if  $\pi_{s} I(z)\in (\Zs)[z]^n$ satisfies system \eqref{KZ}.

\begin{lem}
\label{lem sol ps}
Let $I(z)$ be a solution of system \eqref{KZ} modulo $p^s$.
\begin{enumerate}
\item[(i)]
Let  $t\in \Z_{\geq 0}$.  Then
$p^tI(z)$ is a solution of system \eqref{KZ} modulo $p^{r}$ for any $1\leq r\leq s+t$.

\item[(ii)]
Let  
$f(z)$ be a quasi-constant modulo $p^s$. Then $f(z)I(z)$ 
is a solution of system \eqref{KZ} modulo $p^s$.

\item[(iii)] Let $1\leq t<s$ and $I(z) \in p^t\Z[z]^n$.
Let $f(z)$ be a quasi-constant modulo $p^{s-t}$. Then $f(z)I(z)$ 
is a solution of system \eqref{KZ} modulo $p^s$.

\end{enumerate}
\qed
\end{lem}

\subsection{$p^s$-Hypergeometric solutions}
\label{sec 3.4}

Let $M$ be the least positive integers such that 
\bean
\label{M_i}
M \equiv -\frac{1}{2}
\ \ 
(\on{mod} \,p^s)\,.
\eean
We have 
\bea
M= \frac{p^s-1}2 = \frac{p-1}2\Big(1+p+\dots+p^{s-1}\Big)\,.
\eea
Introduce the {\it master polynomial}
\bean
\label{mp red}
\Phi_{p^s}(x,z) 
&=&
 \prod_{i=1}^n(x-z_i)^{M} \ \in\  \Z[x,z].
\eean
Let
\bean
\label{P}
P_{p^s}(x,z) 
&=&
\Big(\frac {\Phi_{p^s}(x,z)}{x-z_1}, \dots,\frac {\Phi_{p^s}(x,z)}{x-z_n}\Big)\,=\, \sum_i P^{i}_{p^s}(z) \,x^i \,,
\eean
where $P_{p^s}(x,z)$ is a column $n$-vector of polynomials in $x,z_1,\dots,z_n$ and
$P^{i}_{p^s}(z)$ are $n$-vectors of polynomials in $z_1,\dots,z_n$ with coefficients in
$\Z$.
For a positive integer $l$, denote
\bea
I^{[lp^s-1]}_{p^s}(z)\,=\, P^{lp^s-1}_{p^s}(z)\,.
\eea

\begin{thm}
\label{thm Fp} 
For any positive integer $l$, the vector of polynomials $I^{[lp^s-1]}_{p^s}(z) \in \Z[z]^n$
is a solution of  system \eqref{KZ} modulo $p^s$.

\end{thm}

\begin{proof}
We have the following modifications of identities \eqref{i1}, \eqref{i2}\,:
\bean
\label{i3}
M\,
\Big(\frac {\Phi_{p^s}(x,z)}{x-z_1} + \dots + \frac {\Phi_{p^s}(x,z)}{x-z_n}\Big)\,=\,\frac{\der\Phi_{p^s}}{\der x}(x,z)\,,
\eean
\bean
\label{i4}
\Big(\frac{\der }{\der z_i}+ M
\sum_{j\ne i} \frac {\Omega_{i,j}}{z_i-z_j} \Big)\Big(\frac {\Phi_{p^s}(x,z)}{x-z_1}, \dots, \frac {\Phi_{p^s}(x,z)}{x-z_n}\Big) 
\,=\, \frac{\der \Psi_{p^s}^i}{\der x} (x,z),
\eean
where  $\Psi^i_{p^s}(x,z)$ is the column $n$-vector   $(0,\dots,0,-\frac{\Phi_{p^s}(x,z)}{x-z_i},0,\dots,0)$ with 
the nonzero element at the $i$-th place. 
The theorem follows  from identities \eqref{i3}, \eqref{i4}.
\end{proof}

\begin{rem}
In \cite{SV2} it was explained on how to construct polynomial solutions 
modulo $p$ of an arbitrary system differential KZ equations, associated 
with any Kac-Moody algebra and any tensor product of highest weight representations.
The same construction gives polynomial solutions modulo $p^s$.
The details  will be provided elsewhere.

\end{rem}

The range for the index $l$ is defined by  the inequalities 
 $0< lp^s-1\leq n\frac{p^s-1}2$. Hence
$l=1,\dots, g$.
The solutions  $I^{[lp^s-1]}_{p^s}(z)$, $l=1,\dots,g$, 
given by this construction, will be called the 
 {\it  $p^s$-hypergeometric solutions in $\Z[z]^n$}.
 For $t=1,\dots,s-1$ and $l =1,\dots,g$, the vector $p^{s-t} I^{[lp^{t}-1]}_{p^t}(z)$ 
is a solution of system \eqref{KZ} modulo $p^s$,
see Lemma \ref{lem sol ps}.  Such solutions also will 
be called {\it $p^s$-hypergeometric solutions in  $\Z[z]^n$}.

\subsection{Modules}
Consider the increasing filtration
\bean
\label{filt}
0 = \mathcal{M}_{p^s}^0
\subset \mathcal{M}_{p^s}^1\subset \dots\subset \mathcal{M}_{p^s}^{s-1} \subset \mathcal{M}_{p^s}^s = \mathcal{M}_{p^s}\,,
\eean
where
\bean
\label{Def Mst}
\mathcal{M}_{p^s}^t\,&=&\,\Big\{ \pi_{s}\Big(\sum_{r=1}^t\sum_{l=1}^g c_{r,l}(z) \,p^{s-r} I^{[lp^r-1]}_{p^r}(z) \Big)
\ |\ c_{r,l}(z)\in\Z[z]_{p^{r}}\Big\},
\eean
$ t=1,\dots,s$.\
We have $\mathcal{M}_{p^s}\subset (\Zs)[z]^n$. Every element of $\mathcal{M}_{p^s}$ is a polynomial
solution of system \eqref{KZ} with coefficients in $\Zs$, see Lemma \ref{lem sol ps}.
The set $\mathcal{M}_{p^s}$ is a module over the ring $\Z[z]_{p^s}$ of quasi-constants modulo $p^s$, where 
$f(z) \in \Z[z]_{p^s}$ acts by multiplication by $\pi_{s}f(z)$.
Each $\mathcal{M}_{p^s}^t$ is an $\Z[z]_{p^s}$-submodule of $\mathcal{M}_{p^s}$.

Each $\mathcal{M}_{p^s}^t$ is also a module over the larger ring $\Z[z]_{p^t}$ of quasi-constants modulo $p^t$,
where  $f(z) \in \Z[z]_{p^t}$ acts by multiplication by $\pi_{s}f(z)$.

\vsk.2>
The elements of $\mathcal{M}_{p^s}$ will be called the {\it  $p^s$-hypergeometric solutions} in $(\Zs)[z]^n$.

\section{Independence of modules from the choice of $M$}
\label{sec5}

\subsection{More general construction of solutions}
For $i=1,\dots,n$, let $M_i$ be a positive integer  such that 
\bean
\label{M_ii}
M_i \equiv -\frac{1}{2}
\ \ 
(\on{mod} \,p^s)\,.
\eean
Denote $\vec M=(M_1,\dots,M_n)$.
Consider the {\it master polynomial}
\bean
\label{mp red}
\Phi(x,z, \vec M) = \prod_{i=1}^n(x-z_i)^{M_i} \ \in\ \Z[x,z],
\eean
and the Taylor expansion 
\bea
P(x,z,\vec M) =\Big(\frac {\Phi(x,z, \vec M)}{x-z_1}, \dots,\frac {\Phi(x,z, \vec M)}{x-z_n}\Big)
\,=\, \sum_i P^{i}(z, \vec M) \,x^i,
\eea
where $P^{i}(z,\vec M)$ are $n$-vectors of polynomials in $z_1,\dots,z_n$ with coefficients in $\Z$.
For a positive integer $l$, denote
\bea
I^{[lp^s-1]}(z, \vec M)\,=\, P^{lp^s-1}(z, \vec M)\,.
\eea

\begin{thm}
\label{thm p^s} 
For any positive integers $l, t$, $t\leq s$, the vector of polynomials $I^{[lp^t-1]}(z,\vec M) \in \Z[z]^n$
is a solution of  system \eqref{KZ} modulo $p^t$.
\end{thm}

\begin{proof}
The theorem follows from straightforward modifications of identities \eqref{i3}, \eqref{i4}.
\end{proof}

\subsection{More modules}

Consider the increasing filtration
\bean
\label{filT}
0=\mathcal{M}_{p^s}^0(\vec M)  \subset \mathcal{M}_{p^s}^1(\vec M) \subset \dots\subset
 \mathcal{M}_{p^s}^{s-1}(\vec M) \subset \mathcal{M}_{p^s}^s(\vec M) = \mathcal{M}_{p^s}(\vec M)\,,
\eean
where
\bean
\label{Def mst}
\phantom{aaaa}
\mathcal{M}_{p^s}^t(\vec M)\,
&=&\,\Big\{ \pi_{s}\Big(\sum_{r=1}^t\sum_{l\geq 1} c_{r,l}(z) \,p^{s-r} I^{[lp^r-1]}(z,\vec M) \Big)
\ |\ c_{r,l}(z)\in\Z[z]_{p^{r}}\Big\},
\eean
$t=1,\dots,s$. \ 
We have $\mathcal{M}_{p^s}(\vec M)\subset (\Zs)[z]^n$. 
Every element of $\mathcal{M}_{p^s}(\vec M)$ is a polynomial
solution of system \eqref{KZ} with coefficients in $\Zs$,  see Lemma \ref{lem sol ps}.
The set $\mathcal{M}_{p^s}(\vec M)$ is a module over the ring $\Z[z]_{p^s}$ 
of quasi-constants modulo $p^s$, where 
$f(z) \in \Z[z]_{p^s}$ acts by multiplication by $\pi_{s}f(z)$.
Each $\mathcal{M}_{p^s}^t(\vec M)$ 
is an $\Z[z]_{p^s}$-submodule of $\mathcal{M}_{p^s}(\vec M)$.

Each $\mathcal{M}_{p^s}^t(\vec M)$ 
is also a module over the larger ring $\Z[z]_{p^t}$ of quasi-constants modulo $p^t$,
where  $f(z) \in \Z[z]_{p^t}$ acts by multiplication by $\pi_{s}f(z)$.

\begin{thm}
\label{thm ind}

Filtration \eqref{filT} does not depend on the choice of $\vec M=(M_1,\dots,M_n)$, 
satisfying congruences  \eqref{M_ii}.  Moreover, filtration \eqref{filT}
coincides with filtration \eqref{filt}. 

\end{thm}

For $s=1$ the statement is \cite[Theorem 3.1]{SliV}.

\begin{proof}

First we show that
$\mathcal{M}_{p^s}(\vec M)$ and 
  filtration \eqref{filT} do not depend on the choice of $\vec M$.
Let  $\vec M=(M_1,\dots,M_n)$,  $\vec M'=(M_1',\dots,M_n')$ be two vectors satisfying 
congruences  \eqref{M_ii}. We say that $\vec M'\geq \vec M$ if $M_i'\geq M_i$ for all $i$.
The vector $(\frac{p^s-1}2, \dots,\frac{p^s-1}2)$ is the minimal vector with respect to this
partial order. To show that  $\mathcal{M}_{p^s}(\vec M)$
and  filtration   \eqref{filT} do not depend on the choice of $\vec M$
it is enough to show that the filtrations are the same for a vector $\vec M$ and  for
a vector 
$\vec M' = \vec M +(0,\dots,0, \frac{p^s-1}2, 0,\dots,0)$, where the nonzero 
element stays at the $j$-th position for some $j$.
Then
\bea
P(x,z,\vec M') 
&=& P(x,z,\vec M) \cdot (x-z_j)^{p^s} =
P(x,z, \vec M) \sum_{a=0}^{p^s}(-1)^{p^s-a}\binom{p^s}{a}x^a z_j^{p^s-a}.
\eea
 Recall that
$P(x,z,\vec M)=\sum_i P^{i}(z, \vec M) \,x^i,$  $P(x,z,\vec M')=\sum_i P^{i}(z, \vec M') \,x^i,$
For any $r\leq s$ and $l$ we have
\bean
\label{su}
P^{lp^r-1}(z, \vec M') = \sum_{a=0}^{p^s}(-1)^{p^s-a}\binom{p^s}{a} z_j^{p^s-a}P^{lp^r-a-1}(z, \vec M).
\eean
We are interested in this formula, since
$I^{[lp^r-1]}(z, \vec M')=P^{lp^r-1}(z, \vec M')$
is a solution of system \eqref{KZ} modulo $p^r$.

\begin{lem}
\label{lem divi}
Let $b,c\in\Z_{>0}$ be such that  $bp^c \leq p^s$, $p\not\vert\, b$. Then $p^{s-c}$ is the maximal power of $p$ dividing
$\binom{p^s}{bp^c}$.

\end{lem}

\begin{proof} For $a\leq p^s$ we have
$a\binom{p^s}{a} = p^s\binom{p^s-1}{a-1}$,
and $p\not\vert\,\binom{p^s-1}{a-1}$ by Lucas' theorem, \cite{Lu}.
\end{proof}

\begin{lem}
\label{lem su}
Let $a\in\{0,\dots, p^s\}$,\, $a=bp^c$, $p\not\vert \,b$. 
Consider the vector
\bea
 V=(-1)^{p^s-a}\binom{p^s}{a} z_j^{p^s-a}P^{lp^r-a-1}(z, \vec M)
 \eea
 appearing in \eqref{su}.
If $lp^r\leq a$, then $V=0$. For $lp^r> a$, we write $lp^r-a=vp^u$, 
where $u=\min(r,c)$. Then
\bean
\label{Vf}
V \,= \,d(z) \,p^{r-u} I^{[vp^u-1]}(z,\vec M),
\eean
where $d(z) = z_j^{p^s-a}(-1)^{p^s-a}\binom{p^s}{a}/p^{r-u}$ is a quasi-constant modulo
$p^u$.
\end{lem}

\begin{proof} 
The lemma follows from Lemma \ref{lem divi}.
\end{proof}

\begin{cor}
\label{cor 4.4}
For any $r=1,\dots, s$, we have $\mathcal{M}_{p^s}^r(\vec M')\subset \mathcal{M}_{p^s}^r(\vec M)$.
\qed

\end{cor}

\begin{lem}
\label{lem 4.5}
For any $r=1,\dots, s$, we have $\mathcal{M}_{p^s}^r(\vec M')\supset \mathcal{M}_{p^s}^r(\vec M)$.
\end{lem}

\begin{proof}

Let  $w$ be 
the greatest integer such that $wp^r \leq \deg_t P (t,z,\vec M)$.
Then $w' = w+p^{s-r}$ is the greatest integer such that $w'p^r \leq \deg_t P (t,z,\vec M')$.
Comparing the coefficients in \eqref{su} and using Lemma \ref{su}, we observe that
for any 
$l=1,\dots,w$ we have
\bean
\label{syst r}
&&
I^{[(l+p^{s-r})p^r-1]}(z,\vec M') 
=
 I^{[lp^r-1]}(z,\vec M) 
\\
\notag
&&
\phantom{aaaa}
+
 \sum_{m=1}^{l-1} c_{r,m}(z) I^{[mp^r-1]}(z,\vec M) 
+ 
\sum_{k=1}^{r-1} \sum_{m\geq 1} c_{k,m}(z) \,p^{r-k} I^{[mp^k-1]}(z,\vec M),
\eean
where $c_{i,j}(z)\in \Z[z]_{p^i}$. This triangular system of equations with respect to 
$ I^{[lp^r-1]}(z,\vec M)$, $l=1,\dots,w$,
can be written as
\bean
\label{syst r'}
\phantom{aaa}
 I^{[lp^r-1]}(z,\vec M)  =
 \sum_{m\geq 1} c_{r,m}'(z) I^{[mp^r-1]}(z,\vec M') 
+ \sum_{k=1}^{r-1} \sum_{m\geq 1} c_{k,m}'(z) \,p^{r-k} I^{[mp^k-1]}(z,\vec M),
\eean
$l=1,\dots,w$, for suitable $c_{i,j}'(z)\in \Z[z]_{p^i}$. Applying the previous reasoning to the sum
$\sum_{k=1}^{r-1} \sum_{m\geq 1} c_{k,m}(z) \,p^{r-k} I^{[mp^k-1]}(z,\vec M)$, we obtain
\bean
\label{syst r''}
\phantom{aaa}
 I^{[lp^r-1]}(z,\vec M)  =
 \sum_{m\geq 1} c_{r,m}'(z) I^{[mp^r-1]}(z,\vec M') 
+ \sum_{k=1}^{r-1} \sum_{m\geq 1} c_{k,m}''(z) \,p^{r-k} I^{[mp^k-1]}(z,\vec M'),
\eean
$l=1,\dots,w$, for suitable $c_{i,j}''(z)\in \Z[z]_{p^i}$. This proves the lemma.

\end{proof}

\begin{cor}
The module $\mathcal{M}_{p^s}(\vec M)$ and 
filtration \eqref{filT} do not depend on the choice of $\vec M=(M_1,\dots,M_n)$, 
satisfying congruences  \eqref{M_ii}. 
\qed

\end{cor}

\begin{lem}
Let $\vec M^{\on{min}} = (\frac{p^s-1}2, \dots, \frac{p^s-1}2)$. Then
$\mc M_{p^s}(\vec M^{\on{min}}) = \mc M_{p^s}$.

\end{lem}

\begin{proof}  
The proof of the  lemma is  a straightforward modification  of the proof of 
Corollary \ref{cor 4.4} and Lemma \ref{lem 4.5}.
\end{proof}

Theorem \ref{thm ind} is proved.
\end{proof}

\section{Filtrations and homomorphisms}
\label{sec6}

\subsection{Reduction from modulo $p^s$ to modulo $p^{s-m}$}
If $I(z)$ 
is a polynomial solution of system \eqref{KZ} modulo $p^s$, then
$I(z)$  is also a polynomial solution of system \eqref{KZ} modulo $p^{s-m}$ for any $ 1\leq m < s$. 
This defines a map 
\bean
\label{RijM}
&&
\frak r_{s,s-m} : 
\mathcal{M}_{p^s}(\vec M) \to \mathcal{M}_{p^{s-m}}(\vec M),
\\
\notag
&&
 \pi_{s} 
\Big(\sum_{r=1}^s\sum_{l\geq 1} c_{r,l}(z) \,p^{s-r} I^{[lp^r-1]}(z,\vec M)\Big)
\mapsto
 \pi_{s-m} 
\Big(\sum_{r=m+1}^s\sum_{l\geq 1} c_{r,l}(z) \,p^{s-r} I^{[lp^r-1]}(z,\vec M)\Big),
\eean
where $\vec M$ is a vector with coordinates satisfying congruences \eqref{M_ii}.
See these sums  in \eqref{Def mst}.
In the last sum we have
$p^{s-r}I^{[lp^r-1]}(z,\vec M) =p^{s-m-(r-m)}I^{[(lp^{m})p^{r-m}-1]}(z,\vec M)$
and a solution $I^{[lp^r-1]}(z,\vec M)$ modulo $p^r$
 also can  be considered as a solution $I^{[(lp^{m})p^{r-m}-1]}(z,\vec M)$
modulo $p^{r-m}$.

For any $r=1,\dots,s$, the submodule $\mathcal{M}_{p^s}^r(\vec M)\subset \mathcal{M}_{p^s}(\vec M)$
is mapped by $\frak r_{s,s-m}$ to the submodule $\mathcal{M}_{p^{s-m}}^{r-m}(\vec M)$. The induced map
\bean
\label{RMr}
&&
\frak r_{s,s-m} : 
\mathcal{M}_{p^s}^r(\vec M) \to \mathcal{M}_{p^t}^{r-m}(\vec M)
\eean
is a homomorphism of $\Z[z]_{p^r}$-modules. Thus the map \eqref{RijM} is a homomorphism of filtered modules
decreasing the index of filtration by $m$.

By  Theorem \ref{thm ind} we have $\mathcal{M}_{p^s}(\vec M) = \mathcal{M}_{p^s}$. Hence
homomorphism \eqref{RijM} also can  be considered as a homomorphism of filtered modules,
\bean
\label{Rij}
&&
\frak r_{s,s-m} : 
\mathcal{M}_{p^s} \to \mathcal{M}_{p^{s-m}}\,,
\eean
decreasing  the index of filtration by $m$.

It is rather nontrivial to write a formula for this map in terms of the generators
$I^{[lp^r-1]}(z)$ of these modules.

\subsection{Multiplication by $p^m$}

If $I(z)$ is a polynomial solution of system \eqref{KZ} modulo $p^s$, then
for any positive integer $m$ the polynomial $p^mI(z)$ is a polynomial solution 
of system \eqref{KZ} modulo $p^{s+m}$. In particular, multiplication by $p^m$ defines
a map
\bean
\label{p^r}
&&
\frak p_{s,s+m} :  \mathcal{M}_{p^s}^t \to \mathcal{M}_{p^{s+m}}^t\,,
\\
\notag
&&
\pi_{s}\Big(\sum_{r=1}^t\sum_{l=1}^g c_{r,l}(z) \,p^{s-r} I^{[lp^r-1]}_{p^r}(z) \Big)
\mapsto 
\pi_{s+m}\Big(\sum_{r=1}^t\sum_{l=1}^g c_{r,l}(z) \,p^{s+m-r} I^{[lp^r-1]}_{p^r}(z) \Big)
\eean
for any $t=1,\dots,s$.\,
See these sums in \eqref{Def Mst}. Clearly this map
 is an isomorphism of filtered $\Z[z]_{p^t}$-modules.

\subsection{The composition of homomorphisms}
For $m<s$ denote by $\frak c_{s,m}$ the composition $\frak p_{s-m,m} \frak r_{s,s-m}$,
\bean
\label{dec c}
\frak c_{s,m}\ :\ \mc M_{p^s} \to \mc M_{p^s}\,,
\quad
I(z) \mapsto p^mI(z)\,.
\eean
For any $t=1,\dots,s$, this map induces a homomorphism 
$\mc M_{p^s}^t \to \mc M_{p^s}^{t-m}$ of $\Z[z]_{p^t}$-modules.

\vsk.2>
We have $\frak c_{s,m} = (\frak c_{s,1})^m$ for $m<s$ and $\frak c_{s,m} = (\frak c_{s,1})^m=0$ for $m\geq s$.

\vsk.2>
As we know, the module $\mc M_{p^s}$ is generated by the elements
$\pi_s (p^{s-r} I^{[lp^r-1]}_{p^r}(z))$, $r=1,\dots,s$, $l=1,\dots$, $g$. For $l=1,\dots,g$,
we have
\bean
\label{fffc}
\frak c_{s,1}\ :\ \pi_s(I^{[lp^s-1]}(z)) \ \mapsto  \
\pi_s(pI^{[lp^s-1]}(z)) =
\pi_{s} \Big(\sum_{r=1}^{s-1}\sum_{k=1}^g c_{r,k}^{l,s}(z) \,p^{s-r} I^{[kp^r-1]}_{p^r}(z)\Big) 
\eean
for  suitable coefficients $ c_{r,k}^{l,s}(z)\in\Z[z]_{p^{r}}$.

\vsk.2>

 The set of the coefficients $(c_{r,k}^{l,s}(z))_{l,s, r,k}$ 
 determines the homomorphisms $\frak c_{s,m}$ for all $s,m$.   In what follows we shall describe the coefficients
$c_{s-1, k}^{l,s}(z)$ for all  $l,s,k$,  see Theorem \ref{thm mgfc}.

\subsection{Graded modules and homomorphisms}
Denote
\bean
\label{def grM}
\gr \mc M_{p^s}\, =\,\oplus _{t=1}^s\, \gr \mc M_{p^s}^t\,,
\qquad
\gr \mc M_{p^s}^t \,=\, \mc M_{p^s}^t\Big/  \mc M_{p^s}^{t-1}\,.
\eean

\begin{lem}
\label{lem grmo}

For any $t=1,\dots,s$, the action of\, $\Z[z]_{p^t}$ on $\mc M_{p^s}^t$ makes $\gr \mc M_{p^s}^t$ an 
$\F_p[z^{p^t}]$-module. Multiplication by $p$ on $\mc M_{p^s}^t$
induces a homomorphism
\bean
\label{fkcg}
\gr \frak c_{s,1} \ :\ \gr \mc M_{p^s}^t \to \gr \mc M_{p^s}^{t-1}
\eean
of $\F_p[z^{p^t}]$-modules.
\qed

\end{lem}

\section{Coefficients of solutions}
\label{sec7}

\subsection{Homogeneous polynomials}

For $l=1,\dots,g$, the  solution $I^{[lp^s-1]}(z)=(I^{[lp^s-1]}_1$, \dots,  $I^{[lp^s-1]}_n)$ 
is a homogeneous polynomial in $z$ of degree
\bean
\label{deg I}
\delta_l = (2g+1)\,\frac{p^s-1}2\, - \,lp^s
=
(g-l)p^s + \frac{p^s-1}2 - g\,.
\eean
Notice that  $(-1)^{\delta_l}=(-1)^{s\frac{p-1}2 \,+\, l}$\,.

\subsection{Formula for coefficients}
Recall that $M=\frac{p^s-1}2$. Projection of this integer to $\Zs$ is invertible.
Let
\bea
I^{[lp^s-1]}(z) = \sum_{d_1,\dots,d_n} I^{[lp^s-1]}_{d_1,\dots,d_n} z_1^{d_1}\dots z_n^{d_n}\,,
\qquad I^{[lp^s-1]}_{d_1,\dots,d_n} \in \Z^n\,.
\eea

\begin{lem} [{\cite[Lemma 3.1]{V8}}]
\label{lem coef} 
We have
\bean
\label{Coe}
I^{[lp^s-1]}_{d_1,\dots,d_n} \,=\,(-1)^{\dl_l}
\prod_{i=1}^n\binom{M}{d_i}\,
\Big(1 - \frac{d_1}{M}, \dots , 1 - \frac{d_n}{M}\Big).
\eean
The sum of coordinates of this vector is divisible by $p^s$.
\qed
\end{lem}

\begin{lem} [{cf. \cite[Theorem 6.1]{V8}}]
\label{lem lead}
For $l=1,\dots,g$, the leading term of  the $p^s$-hypergeometric solution $I^{[lp^s-1]}(z)$
is
\bean
\label{le term}
I^{[lp^s-1]}_{\frak l}(z)
&=&
(-1)^{\dl_l}
\binom{M}{l}\,\Big(0,\dots,0, \frac lM,1,\dots,1\Big) z_1^M\dots z_{2g-2l}^M z_{2g-2l+1}^{M-l}\,,
\eean
where $0$ is repeated $2g-2l$ times and 1 is repeated $2l$ times.
\end{lem}

\begin{proof} The lemma follows from Lemma \ref{lem coef}.
\end{proof}

\begin{lem}
\label{lem invert}
The projections to $\Zs$ of the integers 
$\binom{M}{l}$, $\binom{M}{l}\frac lM$ are invertible.
\end{lem} 

\begin{proof}
The invertibility of $\binom{M}{l}$ follows from Lucas' theorem, \cite{Lu}.
\end{proof}

\section{Multiplication by $p$ and Cartier-Manin matrix}
\label{sec8}
\subsection{Linear independence}

\begin{lem}
\label{lem li in}
The projections of the $p^s$-hypergeometric solutions  $I^{[lp^s-1]}(z)\in\Z[z]^n$, $l=1,\dots,g$,  to $\F_p[z]^n$  
are linearly independent over $\F_p[z]$,
that is, if 
\bean
\label{lin rel}
\sum_{l=1}^g c_l (z) I^{[lp^s-1]}(z) \ \in \ p\Z[z]^n
\eean
for some $c_l(z)\in \Z[z]$, then all $c_l(z)\in p\Z[z]$.

\end{lem}

\begin{proof} 
Recall that the projection $\Z[z]^n \to \F_p[z]^n$ is denoted by $\pi_1$.
By Lemma \ref{lem invert},  
the leading coefficient of $\pi_1(c_l (z) I^{[lp^s-1]}(z))$
 equals the product of the leading coefficient of $\pi_1(c_l (z))$ 
and the leading coefficient of $\pi_{1} (I^{[lp^s-1]}(z))$, if $\pi_1(c_l(z))$ is nonzero. 
In that case the leading coefficient of $\pi_1(c_l (z) I^{[lp^s-1]}(z))$ is a nonzero multiple of the
nonzero vector 
$\pi_1((0,\dots,0, \frac lM,1,$ \dots, $1))$.

If relation \eqref{lin rel} holds and  some of the coefficients $c_l(z)$ have nonzero projections $\pi_1(c_l(z))$,
then  for several values of such indices $l$ the sum of
the  corresponding leading coefficients has to be equal to zero, which is impossible due to the fact that the vectors
$\pi_1((0,\dots,0, \frac lM,1,\dots,1))$ are linear independent over $\F_p$.
\end{proof}

\begin{cor}
\label{cor lind ps}
The projections of the $p^s$-hypergeometric solutions  $I^{[lp^s-1]}(z)\in\Z[z]^n$, $l=1,\dots,g$,  to $\F_p[z]^n$  
are linearly independent over $\F_p[z^{p^s}]$.
\qed

\end{cor}

Denote by
\bean
\label{grt}
\gr_t\ :\ \mc M_{p^s}^t \ \to \  \gr \mc M_{p^s}^t 
\eean
the natural projection. Then the elements $\gr_t(\pi_s(p^{s-t} I^{[lp^t-1]}(z)))$, $l=1,\dots,g$, 
generate the $\F_p[z^{p^t}]$-module  $\gr \mc M_{p^s}^t$.

\begin{cor}
\label{cor lind pst}
For $t=1,\dots,s$, \,the $\F_p[z^{p^t}]$-module  $\gr \mc M_{p^s}^t$ is a free module of rank $g$ with a basis
$\gr_t(\pi_s(p^{s-t} I^{[lp^t-1]}(z)))$, $l=1,\dots,g$.
\qed

\end{cor}

Denote the basis vectors of the $\F_p[z^{p^t}]$-module  $\gr \mc M_{p^s}^t$ by
\bean
\label{bv}
v_{s,t}^l := \gr_t(\pi_s(p^{s-t} I^{[lp^t-1]}(z))), \qquad l=1,\dots,g.
\eean

\subsection{Cartier-Manin matrices}
Let
\bea
f(x,z) = (x-z_1)\dots(x-z_n),\qquad n=2g+1.
\eea
Consider the hyperelliptic curve $X$ defined by the affine equation
\bea
y^2 = (x-z_1)\dots(x-z_n).
\eea
Consider the space $\Om^1(X)$ of regular 1-forms on $X$ with basis $\frac {x^{i-1}dx}y$,  $i=1,\dots, g$.
Define the Cartier map
$\mc C : \Om^1(X) \to  \Om^1(X)$  as follows. We have
\bea
\frac {x^{i-1}dx}{y} = \frac {x^{i-1}y^{p-1}dx}{y^{p-1}y}
=
\frac {x^{i-1}f(x)^{(p-1)/2}dx}{y^{p}}\,.
\eea
Let
$x^{i-1}f(x,z)^{(p-1)/2} =\sum_j c_i^j(z) x^j$\,.
Define
\bea
\mc C\ :\ \frac {x^{i-1}dx}{y} \mapsto\ \sum_{j=1}^g c_i^{jp-1}(z) \frac {x^{j-1}dx}{y}\,,
\eea
see \cite{AH}. 
The map $\mc C$ is identified with the $g\times g$-matrix $(\mc C_i^j (z))_{i,j=1}^g$\,,
\bea
\mc C_i^j (z) \,= \,c_i^{jp-1}(z).
\eea

\subsection{Matrix of $\gr \frak c_{s,1}$}
Recall that multiplication of solutions by $p$ defines
a map
\bean
\label{kcg}
\gr \frak c_{s,1} \ :\ \gr \mc M_{p^s}^t \to \gr \mc M_{p^s}^{t-1}\,,
\eean
where 
$\gr \mc M_{p^s}^t$ is a free $\F_p[z^{p^t}]$-module with a basis 
$(v_{s,t}^l)_{l=1}^g$ and 
$\gr \mc M_{p^s}^{t-1}$ is a free $\F_p[z^{p^{t-1}}]$-module with a basis 
$(v_{s,t-1}^l)_{l=1}^g$, see \eqref{bv}. The map
\eqref{kcg} is a homomorphism of 
$\F_p[z^{p^t}]$-modules.

\begin{thm}
\label{thm mgfc}
The matrix of $\gr \frak c_{s,1}$  is the Cartier-Manin matrix
$\mc C(z^{p^{t-1}})$,
\bean
\label{grC}
\gr \frak c_{s,1}\ :\ v_{s,t}^l\  \mapsto \ \sum_{m=1}^g v_{s,t-1}^m \mc C_m^l(z^{p^{t-1}}),
\qquad
l=1,\dots,g.
\eean
\end{thm}

\begin{proof}  The problem is to express modulo $p^{s-t+2}\Z[z]^n$ the element
$ p \cdot p^{s-t} I^{[lp^t-1]}_{p^t}(z)$
in terms of the elements $p^{s-t+1} I^{[mp^{t-1}-1]}_{p^{t-1}}(z)$, 
$m=1,\dots,g$. \,In other words, we need to express
$I^{[lp^t-1]}_{p^{t}}(z)$
in terms of $I^{[mp^{t-1}-1]}_{p^{t-1}}(z)$, $m=1,\dots,g,$
modulo $p\Z[z]^n$.
\vsk.2>

By definition, the vector $I^{[lp^t-1]}_{p^{t}}(z)$ 
is the coefficient of $x^{lp^t-1}$ in the Taylor expansion of
the polynomial
\bea
P_{p^t}(x,z) = P_{p^{t-1}}(x,z)  
\prod_{i=1}^n (x-z_i)^{p^{t-1}(p-1)/2} \,,
\eea
while the vector $I^{[mp^{t-1}-1]}_{p^{t-1}}(z)$
 is the coefficient of $x^{mp^{t-1}-1}$ in the Taylor expansion of
the polynomial $P_{p^{t-1}}(x,z)$, see notations  in Section \ref{sec 3.4}. 
We have
\bea
\prod_{i=1}^n (x-z_i)^{p^{t-1}(p-1)/2} \,\equiv \,
\prod_{i=1}^n (x^{p^{t-1}}-z_i^{p^{t-1}})^{(p-1)/2}\,
\qquad
\on{mod}\,p\,.
\eea
Hence $I^{[lp^{t}-1]}_{p^{t}}(z) \equiv \sum_{m=1}^g I^{[mp^{t-1}-1]}_{p^{t-1}}(z)\,
\mc C_m^l(z^{p^{t-1}})$ mod $p$.
Theorem \ref{thm mgfc} is proved.
\end{proof}

\begin{cor}
\label{cor itmp}
The matrix of \
$\gr \frak c_{s,m} \ :\ \gr \mc M_{p^s}^t \to \gr \mc M_{p^s}^{t-m}$
\, is the product of Cartier-Manin matrices 
 $\mc C(z^{p^{t-1}})\mc C(z^{p^{t-2}})\cdots \mc C(z^{p^{t-m}})$.
Moreover, this statement, applied to the map
\bea
\gr \frak c_{s,s-1} \ :\ \gr \mc M_{p^s}^s \to \gr \mc M_{p^s}^{1}=\mc M_{p^s}^{1}\,,
\eea
can be reformulated as follows. For any $l=1,\dots,g$, the solution $I^{lp^s-1}_{p^s}(z)$ 
modulo $p^s$ of system \eqref{KZ}, projected to $\F_p[z]^n$, equals the projection to
$\F_p[z]^n$ of the solution
\bea
\sum_{m_1,\dots,m_{s-1}=1}^g  I^{[m_1p-1]}_p(z)\mc C^{m_2}_{m_{1}}(z^{p})\cdots
\mc C^{m_{s-1}}_{m_{s-2}}(z^{p^{s-2}})
\mc C^l_{m_{s-1}}(z^{p^{s-1}})
\eea
modulo $p$ of system \eqref{KZ}.
\qed
\end{cor}

See these sums in \cite[Section 8]{V5}.

\section{Change of variables}
\label{sec9}
\subsection{Change of the variable $x$}

Change the variable $x$ and set $x=v+z_n$. Then 
\bean
\label{tmp}
\tilde \Phi_{p^s}(v,z)
:=
\Phi_{p^s}(v+z_n,z) 
=
 \Big(\prod_{i=1}^{n-1}(v-(z_i-z_n))\Big)^{(p^s-1)/2}v^{(p^s-1)/2} \,.
\eean
Let
\bean
\label{tP}
\tilde P_{p^s}(v,z) :
&=&
P_{p^s}(v+z_n,z) 
\\
\notag
&=&
\Big(\frac {\tilde \Phi_{p^s}(v,z)}{v-(z_1-z_2)},\dots, \frac {\tilde \Phi_{p^s}(v,z)}{v-(z_{n-1}-z_n)} \,,\,
\frac {\tilde \Phi_{p^s}(v,z)}{v}\Big)\,=\, \sum_i \tilde P^{i}_{p^s}(z) \,v^i \,,
\eean
where $\tilde P^{i}_{p^s}(z)$ are $n$-vectors of polynomials in $z$ with integer coefficients.
 For a positive integer $l$, denote
\bean
\label{Tmd}
\tilde I^{[lp^s-1]}_{p^s}(z)\,:=\, \tilde P^{lp^s-1}_{p^s}(z)\,.
\eean

The polynomial $\tilde I^{[lp^s-1]}_{p^s}(z)$ is nonzero if \ $l=1,\dots,g$. Notice that every polynomial
 $\tilde I^{[lp^s-1]}_{p^s}(z)$ is a function of differences  $z_i-z_n$, $i=1,\dots,n-1$.

\vsk.2>

Consider the increasing filtration
\bean
\label{Tfilt}
0 = \tilde{\mathcal{M}}_{p^s}^0
\subset 
\tilde{\mathcal{M}}_{p^s}^1\subset \dots\subset 
\tilde{\mathcal{M}}_{p^s}^{s-1} 
\subset \tilde{\mathcal{M}}_{p^s}^s = \tilde {\mathcal{M}}_{p^s}\,,
\eean
where
\bean
\label{Def TMs}
\tilde{\mathcal{M}}_{p^s}\,&=&\,\Big\{ \pi_{s} 
\Big(\sum_{r=1}^s\sum_{l=1}^g c_{r,l}(z) \,p^{s-r} \tilde I^{[lp^r-1]}_{p^r}(z)\Big) 
\ |\ c_{r,l}(z)\in\Z[z]_{p^{r}}\Big\},
\\
\label{Def TMst}
\tilde{\mathcal{M}}_{p^s}^t\,&=&\,\Big\{ \pi_{s}\Big(
\sum_{r=1}^t\sum_{l=1}^g c_{r,l}(z) \,p^{s-r} \tilde I^{[lp^r-1]}_{p^r}(z) \Big)
\ |\ c_{r,l}(z)\in\Z[z]_{p^{r}}\Big\},
\eean
$ t=1,\dots,s$.\

\begin{thm}
\label{thm tp^s} 
For any\,  $l$, the vector of polynomials $\tilde I^{[lp^s-1]}_{p^s}(z) \in \Z[z]^n$
is a solution of  system \eqref{KZ} modulo $p^s$. 
For any $t=1,\dots,s$ we have
$\tilde{\mc M}_{p^s}^t= \mc M_{p^s}^t$.

\end{thm}

\begin{proof}
The proof is the same as the proof of Theorem \ref{thm ind} and the proof of \cite[Lemma 5.2]{V5}.
In the proof of Theorem \ref{thm tp^s}  the following Lemma \ref{lem div} 
is used instead of Lemma \ref{lem divi}.

\begin{lem}
\label{lem div} 

Let $r=0,\dots, s-1$ and $m\!\not\vert\, p$, then
$\binom{mp^r +lp^s-1}{lp^s-1}$ is divisible by $p^{s-r}$.
\end{lem}

\begin{proof}  We have $\binom{mp^r +lp^s-1}{lp^s-1}=
\binom{mp^r +lp^s-1}{mp^r} = \frac{lp^s}{mp^r}\,\binom{mp^r +lp^s-1}{mp^r-1}$.
\end{proof}
\end{proof}

\subsection{Change of variables $z$} We introduce the new variables
$u_1,\dots,u_n$ by the formulas\,:
\bean
\label{uz}
\phantom{aaa}
u_1=z_1-z_n, \quad  u_2=\frac{z_{2}-z_n}{z_1-z_{n}},
\ \ \dots \ \
u_{n-1}=\frac{z_{n-1}-z_{n}}{z_{n-2}-z_{n}},
\quad
u_n=z_1+\dots +z_n,
 \eean
or
\bea
z_{i}-z_n = u_1 \cdots u_i, \qquad i=1,\dots,n-1,\qquad
z_1+\dots +z_n = u_n\,.
\eea
For any $l,s$ we denote  $u=(u_1,\dots,u_{n-1})$,
\bean
\label{hat I}
\hat I^{[lp^s-1]}_{p^s}(u)\,: =\, \tilde I^{[lp^s-1]}_{p^s}(z(u))\,.
\eean
Each $\hat I^{[lp^s-1]}_{p^s}(u)$ is an $n$-vector of polynomials in $u$ with coefficients
in $\Z^n$. 

\vsk.2>
Each $\hat I^{[lp^s-1]}_{p^s}(u)$ is a solution of system \eqref{KZ} modulo $p^s$, in which  the change of variables
$z=z(u)$ is performed.

\subsection{Change of variables in the KZ equations} 
It is known that system \eqref{KZ} of the differential
 KZ equations has suitable  asymptotic zones with appropriate local coordinates, in which
the differential KZ equations have singularities only at the coordinate hyperplanes. 
See  a definition of the asymptotic zones, for example, in  \cite{V2}.
The coordinates $u$ defined in \eqref{uz} are local coordinates in one of the asymptotic zones.
In these coordinates,  system  \eqref{KZ}  takes the form,
\bean
\label{Ku}
&&
\frac{\der I}{\der u_n}=0,\qquad 
\frac{\der I}{\der u_i} \,=\, \frac 12 \Big(\frac{\Om_i}{u_i} + \on{Reg}_i(u)\Big) I,
\quad
i=1,\dots,n-1,
\\
\notag
&&
\phantom{aaaaaaaaaaaa}
I_1+\dots+I_n =0,
\eean
where $\Om_i=\sum_{i\leq k<l\leq n}\Om_{k,l}$ and $\on{Reg}_i(u)$ is an $n\times n$-matrix 
depending on $u$ and regular at 
the origin $u=0$. 
The origin is a regular singular point of system \eqref{Ku} and one may expand 
solutions at the origin in  suitable series
in the variables $u$.

\vsk.2>
Any polynomial $\hat I^{[lp^s-1]}_{p^s}(u)$ is a solution of system \eqref{Ku} modulo $p^s$. We will expand the polynomial 
$\hat I^{[lp^s-1]}_{p^s}(u)$ at $u=0$ and show that this expansion has a $p$-adic limit 
as $s\to\infty$.
In that way we will construct  a $g$-dimensional space of $p$-adic solutions of 
system  \eqref{Ku},  which is the same as the original system \eqref{KZ} 
of the differential KZ equations up to the change of variables, $z=z(u)$.

\subsection{Taylor expansion of $\hat I^{[lp^s-1]}_{p^s}(u)$}

Denote
\bean
\label{hmp}
\phantom{aaaa}
\hat \Phi_{p^s}(v,u)
:=
\tilde \Phi_{p^s}(v,z(u)) = 
\left(\prod_{i=1}^{n-1} \Big(v -\prod_{j=1}^iu_j\Big)\right)^{(p^s-1)/2} v^{(p^s-1)/2} ,
\eean
\bean
\label{8.12}
\hat P_{p^s}(v,u) :
=
\Big( \frac {\hat \Phi_{p^s}(v,u)}{v-u_1}, \dots, \frac {\hat \Phi_{p^s}(v,u)}{v-u_1\cdots u_{n-1}},
\frac {\hat \Phi_{p^s}(v,u)}{v}\Big) =
 \sum_i \hat P^{i}_{p^s}(u) \,v^i \,,
\eean
where $\hat P^{i}_{p^s}(u)$ are $n$-vectors of polynomials in $u$ with coefficients in
$\Z$. For a positive integer $l$, we have
\bean
\label{hTmd}
\hat P^{lp^s-1}_{p^s}(u)\,=\,\hat I^{[lp^s-1]}_{p^s}(u)\,,
\eean
where $\hat I^{[lp^s-1]}_{p^s}(u)$ is defined in \eqref{hat I}.

\vsk.2>
For $l=1,\dots,g$, denote
\bean
\label{def uls}
u^{l,s} &=& (-1)^{\delta_l} (u_1\cdots u_{n-2l})^{-l} \,\prod_{i=1}^{n-2l} (u_1\cdots u_i)^{\frac{p^s-1}2}\,,
\eean
or
\bea
&&
u^{g,s} = (-1)^{\delta_g}u_1^{\frac{p^s-1}2 - g}, \quad
u^{g-1,s} = (-1)^{\delta_{g-1}}
u_1^{3\frac{p^s-1}2 - g+1}u_2^{2\frac{p^s-1}2 - g+1}u_3^{\frac{p^s-1}2 - g+1},
\dots
\\
&&
\phantom{aaaa}
u^{1,s} = (-1)^{\delta_1}u_1^{(n-2)\frac{p^s-1}2 - 1}u_2^{(n-3)\frac{p^s-1}2 - 1}\cdots
u_{n-2}^{\frac{p^s-1}2 - 1},
\eea
see $(-1)^{\delta_l}$
in \eqref{deg I}.
Denote
\bean
\label{leac}
C^{l,s} =
\,\binom{\frac{p^s-1}2}{l} 
\Big(0,\dots,0, \frac {2l} {p^s-1},\,1,\dots,1\Big) \,,
\eean
where $0$ is repeated $2g-2l$ times and 1 is repeated $2l$ times,
 cf. formula \eqref{le term}.

\begin{thm}
\label{thm hat I}
For $l=1,\dots,g$, the polynomial $\hat I^{[lp^s-1]}_{p^s}(u)$ has the following form,
\bean
\label{hIf}
\hat I^{[lp^s-1]}_{p^s}(u) 
&=&
 u^{l,s} T^{l,s}(u)\,,
\\
\notag
T^{l,s}(u) 
&=& (T^{l,s}_1(u),\dots,T^{l,s}_n(u)),
\eean
with coordinates $T^{l,s}_j$ defined  as follows.
If $j=1,\dots, n-1$, then
\bean
\label{tj}
T^{l,s}_j
&=& 
 \,u_{j+1}\cdots u_{n-2l}\,\, {\sum}^{l,s,j} \,\binom{\frac{p^s-3}2}{a_{j}}
\,\prod_{i=1, \,i\ne j}^{n-1} \binom{\frac{p^s-1}2}{a_i}
\\
\notag
&\times& \,
 \prod_{i=1}^{n-2l-1}(u_{i+1}\cdots u_{n-2l})^{a_i} \prod_{i=1}^{2l-1}(u_{n-2l+1}\cdots u_{n-2l+i})^{a_{n-2l+i}}\,,
\eean
where the summation ${\sum}^{l,s,j}$
is over all $a_1,\dots,a_{n-1}\in\Z$,  $0\leq a_i\leq \frac{p^s-1}2$,
such that
$a_1+\dots+a_{n-2l}=a_{n-2l+1}+\dots+a_{n-1} + l-1$, if $j\leq n-2l$;
and such that
$a_1+\dots+a_{n-2l}=a_{n-2l+1}+\dots+a_{n-1} + l$, if $ n-2l<j\leq n-1$;
\bean
\label{tn}
&&
\phantom{aaa}
\\
\notag
&&
T^{l,s}_n = \, {\sum}^{l,s,n} \,\,\prod_{i=1}^{n-1}
\binom{\frac{p^s-1}2}{a_i}
 \prod_{i=1}^{n-2l-1}(u_{i+1}\cdots u_{n-2l})^{a_i} \prod_{i=1}^{2l-1}(u_{n-2l+1}\cdots u_{n-1})^{a_{n-2l+i}}\,,
\eean
where the summation ${\sum}^{l,s,n}$
is over all $a_1,\dots,a_{n-1}\in\Z$,  $0\leq a_i\leq \frac{p^s-1}2$,
such that 
$a_1+\dots+a_{n-2l}=a_{n-2l+1}+\dots+a_{n-1} + l$.

The constant term of $T^{l,s}(u)$
equals $C^{l,s}$. 

\end{thm}

Notice that the factor $u_{j+1}\cdots u_{n-2l}$ in \eqref{tj} equals 1 if $j\geq n-2l$.

\begin{proof}

Make the change of variables $v=w u_1\cdots u_{n-2l}$ in \eqref{8.12},
\bea
P_{p^s}^\circ (w,u) :=\hat P_{p^s}(w u_1\cdots u_{n-2l},u) 
=
 \sum_i \hat P^{i}_{p^s}(u) \,(w u_1\cdots u_{n-2l})^i =:
\sum_i P^{\circ,i}_{p^s}(u) \, w ^i\,.
\eea
Hence
\bea
\hat I^{[lp^s-1]}_{p^s}(u) \,= \,\big(u_1\cdots u_{n-2l}\big)^{-(lp^s-1)}\,P^{\circ,lp^s-1}_{p^s}(u)\,.
\eea
We transform the factors in the polynomial 
$P_{p^s}^\circ (w,u)$ as follows. For any positive integers $k$ and $i \leq n-2l$  we write
\bean
\label{exp1}
(w u_1\cdots u_{n-2l}-u_1\cdots u_i)^k 
&=& (u_1\cdots u_i)^k (wu_{i+1}\cdots u_{n-2l}- 1)^k
\\
\notag
&=&
(u_1\cdots u_i)^k \sum_{a=0}^k(-1)^{k-a}
\binom{k}{a}(wu_{i+1}\cdots u_{n-2l})^a,
\eean
and if $i > n-2l$,   we write
\bean
\label{exp2}
(w u_1\cdots u_{n-2l}-u_1\cdots u_i)^k 
&=& (w u_1\cdots u_{n-2})^k (1-u_{n-2l+1}\cdots u_i/w)^k
\\
\notag
&=&
 (w u_1\cdots u_{n-2})^k 
 \sum_{a=0}^k(-1)^a
\binom{k}{a}(u_{n-2l+1}\cdots u_i/w)^a.
\eean
Notice that for factors in \eqref{8.12}, we have $k=\frac{p^s-1}2$ or $k=\frac{p^s-1}2-1$. This explains the binomial coefficients in
\eqref{tj} and \eqref{tn}.

\vsk.2>

We prove formula \eqref{tj} for $j=1$,  the proof for other values of $j$ is similar.

\vsk.2>
Our goal is to calculate the first coordinate of the vector
$\big(u_1\cdots u_{n-2}\big)^{-(lp^s-1)}\,P^{\circ,lp^s-1}_{p^s}(u)$. 
That is we need to calculate the coefficient of $w^{lp^s-1}$ in
\bea
&&
\big(u_1\cdots u_{n-2}\big)^{-(lp^s-1)} 
u_1^{\frac{p^s-1}2-1}\,(wu_2\cdots u_{n-2l} -1)^{\frac{p^s-1}2-1}
(u_1u_2)^{\frac{p^s-1}2}(wu_3\cdots u_{n-2l} -1)^{\frac{p^s-1}2}\cdots
\\
&&
\phantom{aaaa}
\cdots
(u_1\cdots u_{n-2l})^{\frac{p^s-1}2}(w -1)^{\frac{p^s-1}2}
(wu_1\cdots u_{n-2l})^{\frac{p^s-1}2} (1 -u_{n-2l+1}/w)^{\frac{p^s-1}2}
\cdots
\\
&&
\phantom{aaaaaaaa}
\cdots
(wu_1\cdots u_{n-2l})^{\frac{p^s-1}2}(1 -u_{n-2l+1}\cdots u_{n-1}/w)^{\frac{p^s-1}2}
(wu_1\cdots u_{n-2l})^{\frac{p^s-1}2}\,,
\eea
which is the same as the coefficient of $w^{l-1}$ in
\bean
\label{8.21}
&&
\, u_2\cdots u_{n-2l}\,u^{l,s}\,
(wu_2\cdots u_{n-2l} -1)^{\frac{p^s-1}2-1}
(wu_3\cdots u_{n-2l} -1)^{\frac{p^s-1}2}\cdots
\\
&&
\notag
\phantom{aaaa}
\cdots
(w -1)^{\frac{p^s-1}2}
(1 -u_{n-2l+1}/w)^{\frac{p^s-1}2}
\cdots
(1 -u_{n-2l+1}\cdots u_{n-1}/w)^{\frac{p^s-1}2}\,.
\eean
Expanding the binomials we obtain formula \eqref{tj} for $j=1$.

\vsk.2>
The constant term of $T^{l,s}(u)$ is given by
the summands in \eqref{tj} and \eqref{tn}, corresponding to 
$a_1=\dots=a_{n-2l-1}=a_{n-2l+1}=\dots = a_{n-1}=0$  and $j=n-2l, \dots, n$.
Theorem \ref{thm hat I} is proved.
\end{proof}

\subsection{Taylor expansion of holomorphic solutions}
\label{sec 9.5}

Recall the multivalued holomorphic solutions of system \eqref{KZ} described in Section \ref{sec2},
\bea
I^{(\ga)}(z) 
=
\int_\ga\Big(\frac {\Phi(x,z)}{x-z_1}, \dots,\frac {\Phi(x,z)}{x-z_n}\Big) dx\,.
\eea
We make the same changes of variables in the integrals $I^{(\ga)}(z)$ as we did in
the previous sections. Namely, first we change the integration variable $x$ and set  $x=v+z_n$,
then we make the change of variables $z$ and set
$z=z(u)$. The resulting integral is
\bea
\hat I^{(\ga)}(u) 
&=&
\int_\ga\Big(\frac {\hat \Phi(v, u)}{v-u_1}, \dots,\frac {\hat \Phi(v,u)}{v}\Big) dv\,,
\eea
where $\hat \Phi(v, u)
=
\left(\prod_{i=1}^{n-1} \Big(v -\prod_{j=1}^iu_j\Big)\right)^{-1/2}v^{-1/2}$.

\vsk.2>
For $l=1,\dots, g$, we change the integration variable $v$ and set $v=w u_1\dots u_{n-2l}$. 
Then
\bea
&&
\hat \Phi(w u_1\dots u_{n-2l}, u) = e^{(l-n/2)\pi i}\,
(u_1\cdots u_{n-2l})^{-l} \prod_{i=1}^{n-2l} (u_1\cdots u_i)^{-1/2}
\\
&&
\phantom{a}
\times\, 
\Big((1-wu_2\cdots u_{n-2l}) (1-wu_3\cdots u_{n-2l})\cdots (1-w)
\\
&&
\phantom{aaa}
\times\,
(1 -u_{n-2l+1}/w)
\cdots
(1 -u_{n-2l+1}\cdots u_{n-1}/w)\Big)^{-1/2} \, w^{-l}\,.
\eea

 Choose
the integration cycle $\ga=\gamma_l$ to be the circle $|w|=1/2$ oriented counter-clockwise.
We assume that all the variables $u_2,\dots,u_{n-1}$ lie inside the circle. We fix the branch of 
the function
\bean
\label{arg}
&&
\Big((1-wu_2\cdots u_{n-2l}) (1-wu_3\cdots u_{n-2l})\cdots (1-w)
\\
\notag
&&
\phantom{aaa}
\times\,
(1 -u_{n-2l+1}/w)
\cdots
(1 -u_{n-2l+1}\cdots u_{n-1}/w)\Big)^{-1/2} 
\eean
over the circle by choosing the argument of the function in \eqref{arg}
 at $w=1/2$, $u_2=\dots =u_{n-1}=0$  to be $0$.
 We multiply the circle with the chosen branch of the integrand
by $\frac { e^{n\pi i/2}}{2\pi i}$. This finishes the description of $\ga_1$.

The resulting integral is
\bea
\hat I^{(\ga_l)}(u) 
&=&
\frac { e^{n\pi i/2}}{2\pi i}\,\int_{|w|=1/2}\Big(\frac {\hat \Phi(w u_1\dots u_{n-2l}, u)}{w u_1\dots u_{n-2l}-u_1}, 
\dots,\frac {\hat \Phi(w u_1\dots u_{n-2l},u)}{w u_1\dots u_{n-2l}}\Big) u_1\dots u_{n-2l}\,dw\,.
\eea
Denote
\bean
\label{U l}
u^l:= (u_1\cdots u_{n-2l})^{-l} \prod_{i=1}^{n-2l} (u_1\cdots u_i)^{-1/2}\,.
\eean

\begin{thm}
\label{thm hot I}
For $l=1,\dots,g$, the function $\hat I^{(\ga_l)}(u)$ has the following form,
\bean
\label{hoIf}
\hat I^{(\ga_l)}(u) 
&=&
u^l\, T^{l}(u)\,,\qquad
T^{l}(u) 
= (T^{l}_1(u),\dots,T^{l}_n(u)),
\eean
with coordinates $T^{l}_j$ defined as follows.
If $j=1,\dots, n-1$, then
\bean
\label{Tj}
T^{l}_j
&=& 
 \,u_{j+1}\cdots u_{n-2l}\,\, {\sum}^{l,j} \,\binom{\frac{-3}2}{a_{j}}
\,\prod_{i=1, \,i\ne j}^{n-1} \binom{\frac{-1}2}{a_i}
\\
\notag
&\times& \,
 \prod_{i=1}^{n-2l-1}(u_{i+1}\cdots u_{n-2l})^{a_i} \prod_{i=1}^{2l-1}(u_{n-2l+1}\cdots u_{n-2l+i})^{a_{n-2l+i}}\,,
\eean
where the summation ${\sum}^{l,j}$
is over all $a_1,\dots,a_{n-1}\in\Z_{\geq 0}$
such that
$a_1+\dots+a_{n-2l}=a_{n-2l+1}+\dots+a_{n-1} + l-1$, if $j\leq n-2l$;
and such that
$a_1+\dots+a_{n-2l}=a_{n-2l+1}+\dots+a_{n-1} + l$, if $ n-2l<j\leq n-1$;
\bean
\label{Tn}
&&
\phantom{aaa}
T^{l}_n = \, {\sum}^{l,n} \,\,\prod_{i=1}^{n-1}
\binom{\frac{-1}2}{a_i}
 \prod_{i=1}^{n-2l-1}(u_{i+1}\cdots u_{n-2l})^{a_i} \prod_{i=1}^{2l-1}(u_{n-2l+1}\cdots u_{n-1})^{a_{n-2l+i}}\,,
\eean
where the summation ${\sum}^{l,n}$
is over all $a_1,\dots,a_{n-1}\in\Z_{\geq 0}$
such that 
$a_1+\dots+a_{n-2l}=a_{n-2l+1}+\dots+a_{n-1} + l$.

The power series $T^l(u)$ converges in the polydisc $\{(u_2,\dots,u_{n-1})\in\C^{n-1}\ |\
|u_i|<1,\, i=2,\dots,n-1\}$.
\end{thm}

\begin{proof} 

We prove formula \eqref{tj} for $j=1$,  the proof for other values of $j$ is similar.

\vsk.2>
The function $T^{l}_1(u)$ equals
\bea
&&
\frac { e^{n\pi i/2}}{2\pi i}\, \int_{|w|=1/2}
u_1^{\frac{-3}2}\,(wu_2\cdots u_{n-2l} -1)^{\frac{-3}2}
(u_1u_2)^{\frac{-1}2}(wu_3\cdots u_{n-2l} -1)^{\frac{-1}2}\cdots
\\
&&
\phantom{a}
\cdots
(u_1\cdots u_{n-2l})^{\frac{-1}2}(w -1)^{\frac{-1}2}
(wu_1\cdots u_{n-2l})^{\frac{-1}2} (1 -u_{n-2l+1}/w)^{\frac{-1}2}
\cdots
\\
&&
\phantom{aa}
\cdots
(wu_1\cdots u_{n-2l})^{\frac{-1}2}(1 -u_{n-2l+1}\cdots u_{n-1}/w)^{\frac{-1}2}
(wu_1\cdots u_{n-2l})^{\frac{-1}2}\,u_1\cdots u_{n-2l}\,dw 
\eea
\bea
&&
=\, \frac {(-1)^{l-1}}{2\pi i}\, 
\, u_1^{-1}(u_1\cdots u_{n-2l})^{-l+1} \prod_{i=1}^{n-2l} (u_1\cdots u_i)^{\frac{-1}2}
\\
&&
\phantom{a}
\times\, \int_{|w|=1/2} 
(1-wu_2\cdots u_{n-2l} )^{\frac{-3}2}
\Big((1-wu_3\cdots u_{n-2l})\cdots (1-w)
\\
&&
\phantom{aaa}
\times\,
(1 -u_{n-2l+1}/w)
\cdots
(1 -u_{n-2l+1}\cdots u_{n-1}/w)\Big)^{\frac{-1}2} \,\frac{dw}{w^l}\,.
\eea
Expanding the binomials we obtain formula \eqref{Tj} for $j=1$.

The convergence property is clear.
\end{proof}

\subsection{Formal solutions over $\Q_p$ and truncation}

For $l=1,\dots,g$, the formal series  $\hat I^{(\ga_l)}(u)$ is a formal solution of system
\eqref{KZ}, in which  the change of variables
$z=z(u)$ is performed and which is considered over the field $\Q_p$
 of $p$-adic numbers.

\begin{lem}
\label{lem padind}
The formal series  $\hat I^{(\ga_l)}(u)$, $l=1,\dots,g$, 
are linear independent over the 
\linebreak
field $\Q_p$.

\end{lem}

\begin{proof} The proof follows from the fact that the  monomials $u^l$, $l=1,\dots,n$, are linear independent over
$\Q_p$.
\end{proof}

\subsection{Example $n=3$} 
\label{sec Exn3}

In this case we have $g=1$, $u_1=z_1-z_3$, 
$u_2=\frac{z_2-z_3}{z_1-z_3}$,
\bean
\label{hat g=1 s}
\hat I^{(\ga_1)}(u_1, u_2) 
&=&
 u_1^{-3/2}  \sum_{a=0}^\infty \Big(\binom{-\frac32}{a}\binom{-\frac12}{a},
\binom{-\frac12}{a+1}\binom{-\frac32}{a},
\binom{-\frac12}{a+1}\binom{-\frac12}{a}\Big) \,u_2^a
\\
\notag
&=&
 u_1^{-3/2}  \sum_{a=0}^\infty \binom{-\frac12}{a+1}\binom{-\frac32}{a}
\Big(\frac{a+1}{-1/2-a},1,\frac{-1/2}{-1/2-a}\Big) u_2^a
\eean
\bean
\label{g=1 s}
&&
\hat I^{[p^s-1]}(u_1,u_2)  =
\\
\notag
&&
=\,
 (-1)^{\frac{p^s-3}2} u_1^{\frac{p^s-3}2} 
 \sum_{a=0}^{\frac{p^s-1}2}\Big(\binom{\frac{p^s-3}2}{a}\binom{\frac{p^s-1}2}{a},
\binom{\frac{p^s-1}2}{a+1}\binom{\frac{p^s-3}2}{a},
\binom{\frac{p^s-1}2}{a+1}\binom{\frac{p^s-1}2}{a}\Big) \,u_2^a
\\
\notag
&&
 =(-1)^{\frac{p^s-3}2} u_1^{\frac{p^s-3}2} 
\sum_{a=0}^{\frac{p^s-1}2} \binom{\frac{p^s-1}2}{a+1}\binom{\frac{p^s-3}2}{a}
\Big(\frac {a+1}{(p^s-1)/2-a},1,\frac {(p^s-1)/2}{(p^s-1)/2-a}\Big) u_2^a\,
\eean

\section{$p$-Adic convergence} 
\label{sec10}

Consider the field  $\Q_p$  with the standard $p$-adic norm $|t|_p,$ \,$t\in\Q_p$.
In this section we consider the polynomial solutions $((-1)^{\delta_l}T^{l,s}(u))_{s=1}^\infty$
and the formal solutions $\hat I^{(\ga_l)}(u)$ as functions 
\linebreak
on $\Q_p^{n-1}$.

Recall that $\Z_p\subset \Q_p$ denotes the ring of $p$-adic integers.

\subsection{Teichmuller representatives}

For  $t\in\Z_p$ there exists the unique solution $\om(t)\in \Z_p$ of the equation 
$\om(t)^p=\om(t)$ that is congruent to $t$ modulo $p$. The element $\om(t)$
is called the {\it Teichmuller representative}. It also can be defined by $\om(t) = \lim_{s\to\infty} t^{p^s}$.
The {\it Teichmuller character} is the homomorphism 
\bea
\F_p^\times \to \Z_p^\times, \qquad \al \mapsto \om(\al), \qquad \al\in \F_p^\times.
\eea

For $\al \in\F_p$, $r>0$, define the disc
\bea
 D_{\al,r} = \{ t\in \Z_p\  |  \  |t-\om(\al)|_p < r \}.
\eea
The space $\Z_p$ is the disjoint union of the  discs $D_{\al,1}$, $\al \in \F_p$.
The function $\om : \Z_p\to\Z_p$, $t\mapsto \om(t)$, is a locally constant function 
equal to $\om(\al)$ on the disc $D_{\al,1}$. 

For a subset $S\subset \Z_p$ and a function $f:S\to \Z_p$ define the norm
\bea
\parallel f\parallel = \sup_{t\in S} |f(t)|_p\,.
\eea

\begin{lem}
\label{lem xps} 
For any $\al\in \F_p$ the sequence of polynomial functions
$(x^{p^s})_{s=1}^\infty$ uniformly converges on $D_{\al,1}$ to the 
constant function $\om(\al)$.
\end{lem}

\begin{proof}
We use the ``fundamental inequality'' from \cite[II.4.3]{Ro}: if $|t|_p\leq 1$, then
$|(1+t)^{p^s}-1|_p\leq |t|_p\cdot\max(|t|_p,1/p)^s$. Now 
let  $t\in   D_{\al,1}$\,.\ Then $t^{p-1} = 1 + t_1$, $|t_1|_p\leq 1/p$.
We have  $|t^{p^{s+1}}-t^{p^s}|_p =|t^{p^s}|_p |t^{(p-1)p^s}-1|_p$
$\leq |(1+t_1)^{p^s}-1|_p \leq 1/p^{s+1}$.

For positive integers $s_1,s_2$ and $t\in  D_{\al,1}$ we have
$|t^{p^{s_1+s_2}}-t^{p^{s_1}}|_p 
= |t^{p^{s_1+s_2}}- t^{p^{s_1+s_2-1}} +t^{p^{s_1+s_2-1}}
+\dots +
  t^{p^{s_1+1}}-
  t^{p^{s_1}}|_p \leq 1/p^{s_1+1}$. Hence 
the sequence $(x^{p^s})_{s=1}^\infty$ is a Cauchy sequence. 

For $t\in D_{0,1}$, we have  $|t^{p^s}|_p\leq 1/p^s$.
For $t_1,t_2 \in D_{\al,1},$ $\al\ne 0$,  we have $t_1/t_2 =1+t$ with
$|t|_p\leq 1/p$ and
$|t_1^{p^s}-t_2^{p^s}|_p = |(1+t)^{p^s}-1|_p \leq 1/p^{s+1}$.
The lemma is proved.
\end{proof}

For $\al\in \F_p$ consider the sequence of polynomial functions
$(x^{(p^s-1)/2})_{s=1}^\infty$ on $D_{\al,1}$.
This sequence uniformly converges to $0$ on the disc $D_{0,1}$.

\vsk.2>
Let $\al^{(p-1)/2}= 1$. Let $\beta \in \F_p$ be such that $\beta^2=\al$.
The function $D_{\beta,1}\to  D_{\al,1}$, $t\mapsto t^2$, is  an analytic diffeomorphism.
The inverse function $D_{\al,1}\to D_{\beta,1}$ will be denoted by $x^{1/2}$. 
There are two square roots $\pm x^{1/2}$. The root $x^{1/2}$ 
corresponds to the chosen $\beta\in\F_p$ and the root $-x^{1/2}$ 
corresponds to $-\beta\in\F_p$. 

\vsk.2>

We change the variable $x$,  set $x=y^2$, and lift the sequence 
$(x^{(p^s-1)/2})_{s=1}^\infty$ to the sequence $(y^{p^s-1})_{s=1}^\infty$ of polynomial functions
on $D_{\beta,1}$.

\begin{lem}
\label{lem x=y2}
The sequence of polynomial functions
$(y^{p^s-1})_{s=1}^\infty$ uniformly converges on 
$D_{\beta,1}$ to the function $\om(\beta)/y$.
In other words, the sequence of polynomial functions
$(x^{(p^s-1)/2})_{s=1}^\infty$ uniformly converges on 
$D_{\al,1}$ to the function $\om(\beta) x^{-1/2}$.

\end{lem}

\begin{proof}
The lemma follows from Lemma \ref{lem xps}.
\end{proof}

\vsk.2>
Let $\al^{(p-1)/2}=-1$. Then 
$\om(\al)^{(p^{s}-1)/2}= (-1)^{1+p+\dots+p^{s-1}}=(-1)^s$ and
the sequence $(x^{(p^s-1)/2})_{s=1}^\infty$ has no limit on   $D_{\al,1}$.

\subsection{Approximation of binomial coefficients}
It is known that $\binom{-\frac{1}2}{a}, \binom{-\frac{3}2}{a}\in\Z_p$ for $a\in\Z_{\geq 0}$.

\begin{lem}
\label{lem BIN}
Let $l_1\geq 0,\,l_2\geq 0$\,
be integers. Then 
there exists an integer $s_0\geq 0,$ such that 
for any integer $s\geq s_0$ and any integer $a$ with 
 $\frac{p^s-1}2-l_1\geq l_2+ a\geq 0$ we have
\bean
\label{BIN}
\phantom{aaa}
\left|\binom{-\frac12-l_1}{l_2+a} - \binom{\frac{p^s-1}2-l_1}{l_2+a} \right|_p\leq 1/p^{s-d-a}\,,
\qquad \on{where}\quad
d=l_1+l_2-1/2.
\eean

\end{lem}

\begin{proof}
We have
\bea
\binom{-\frac12-l_1}{l_2+a}
& =& \frac1{(-2)^{l_2+a}(l_2+a)!}\, \prod_{k=1}^{l_2+a} (2(l_1+k)-1),
\\
\binom{\frac{p^s-1}2-l_1}{l_2+a} 
&=&
 \frac1{(-2)^{l_2+a}(l_2+a)!} \,\prod_{k=1}^{l_2+a} (2(l_1+k)-1-p^s)\,.
\eea
The $p$-adic norm of $\binom{-\frac12-l_1}{l_2+a}$ is $\leq 1$.
The difference $\binom{-\frac12-l_1}{l_2+a} - \binom{\frac{p^s-1}2-l_1}{l_2+a}$ is the sum of products
$\frac1{(-2)^{l_2+a}(l_2+a)!} \,\prod_{k=1}^{l_2+a} (2(l_1+k)-1)$ in each of which at least one of the factors
$2(l_1+k)-1$ is replaced with $-p^s$. We prove that even one such replacement 
implies that this summand of the difference has $p$-adic norm
$\leq  1/p^{s+1/2-l_1-l_2-a}$.

Indeed, let $2(l_1+k)-1 = bp^c$, $p \not\vert b$. We have
$2(l_1+k)-1\leq 2(l_1+l_2+a)-1\leq 2(l_1 + \frac{p^s-1}2) -1$.
Hence $bp^c \leq 2l_1 -2+p^s$. Hence for any $s$, large enough, we have $c\leq s$,
and the  replacement of
$2(l_1+k)-1$ with $-p^s$ makes the norm of that summand $\leq 1/p^{s-c}$.

We also have $bp^c\leq 2(l_1+l_2+a)-1$. Hence
$c\leq p^{c-1}<\frac{p^c}{2}\leq \frac{bp^c}{2}\leq a+l_1+l_2-\frac{1}{2}$.
This shows that each summand
has the $p$-adic norm $1/\leq p^{s-c}\leq 1/p^{s+1/2-l_1-l_2-a}$.
The lemma is proved.
\end{proof}

\subsection{Example $n=3$, continuation}
\label{sec n=3 c}

Consider the formal power series
\bean
\label{g=1T}
T^1(x) 
&=&
 \sum_{a=0}^\infty \Big(\binom{-\frac32}{a}\binom{-\frac12}{a},
\binom{-\frac12}{a+1}\binom{-\frac32}{a},
\binom{-\frac12}{a+1}\binom{-\frac12}{a}\Big) \,x^a
\eean
in \eqref{hat g=1 s} and the sequence of polynomials 
\bean
\label{g=1P}
T^{1,s}(x) 
=\,
 \sum_{a=0}^{\frac{p^s-1}2}\Big(\binom{\frac{p^s-3}2}{a}\binom{\frac{p^s-1}2}{a},
\binom{\frac{p^s-1}2}{a+1}\binom{\frac{p^s-3}2}{a},
\binom{\frac{p^s-1}2}{a+1}\binom{\frac{p^s-1}2}{a}\Big) \,x^a
\eean
in \eqref{g=1 s}  as functions on $\Z_p$.

\begin{prop}
\label{prop convP g=1}

The power series $T^1(x)$ uniformly convergence  on $D_{0,1}$.
The sequence of polynomial functions
$(T^{1,s}(x))_{s=1}^\infty$ uniformly converges on $D_{0,1}$ to the 
 function $T^1(x)$.

\end{prop}

\begin{proof} 
The fact that the binomials 
 $\binom{-\frac{1}2}{a}, \binom{-\frac{3}2}{a}$ are $p$-adic integers
  implies the uniform convergence of the power series $T^1(x)$ on $D_{0,1}$.

Let us write $T^1(x) = \sum_{a=0}^\infty T^1_a x^a$
and
 $T^{1,s}(x) = \sum_{a=0}^{\frac{p^s-1}2}
 T^{1,s}_a x^a$, where $T^{1}_a, T^{1,s}_a \in \Z_p^3$. 
  Then 
 $T^1(x) - T^{1,s}(x)= \sum_{a=\frac{p^s+1}2}^\infty T^{1}_a \,x^a +
 \sum_{a=0}^{\frac{p^s-1}2} (T^{1}_a-T^{1,s}_a) \,x^a$.
Clearly the $p$-adic norm of the first sum is $\leq 1/p^{\frac{p^s-1}2}$. 
 By Lemma \ref{lem BIN} if $s$ is big enough, then
 for each summand of the second sum and $t\in D_{0,1}$ we have
 $|(T^{1}_a-T^{1,s}_a) \,t^a|_p\leq 1/p^{s-d}$ for some $d$ independent of $s$ and of the summand. 
 This proves the proposition.
 \end{proof}

Consider the formal series 
\bean
\label{hat g=1 sc}
\hat I^{(\ga_1)}(u_1, u_2) 
&=&
 u_1^{-3/2}  \sum_{a=0}^\infty \Big(\binom{-\frac32}{a}\binom{-\frac12}{a},
\binom{-\frac12}{a+1}\binom{-\frac32}{a},
\binom{-\frac12}{a+1}\binom{-\frac12}{a}\Big) \,u_2^a
\eean
and the sequence of polynomials
\bean
\label{g=1 sc}
&&
 (-1)^{\frac{p^s-3}2} \hat I^{[p^s-1]}(u_1,u_2)  =
\\
\notag
&&
=\,
u_1^{\frac{p^s-3}2} 
 \sum_{a=0}^{\frac{p^s-1}2}\Big(\binom{\frac{p^s-3}2}{a}\binom{\frac{p^s-1}2}{a},
\binom{\frac{p^s-1}2}{a+1}\binom{\frac{p^s-3}2}{a},
\binom{\frac{p^s-1}2}{a+1}\binom{\frac{p^s-1}2}{a}\Big) \,u_2^a
\eean
 as functions on $D_{\al,1}\times D_{0,1}$, where
$\al = \beta^2$ for some $\beta\in\F_p$. Then the function $u_1^{1/2} : 
D_{\al,1}\to D_{\beta,1}$ is well-defined
and the series $\hat I^{(\ga_1)}(u_1, u_2) $ is a well-defined function on  $D_{\al,1}\times D_{0,1}$.

\begin{thm}
\label{thm 10.3}

The sequence of polynomial functions 
$\big( (-1)^{\frac{p^s-3}2} \hat I^{[p^s-1]}(u_1,u_2)\big)_{s=1}^\infty$ uniformly converges on 
$D_{\al,1}\times D_{0,1}$ to the function 
$\om(\beta) \hat I^{(\ga_1)}(u_1, u_2) $\,.

\end{thm}

\begin{proof}
The theorem follows from Lemma \ref{lem x=y2} and Proposition \ref{prop convP g=1} .
\end{proof}

\subsection{$p$-Adic convergence for arbitrary $n$}
\label{sec nab}

Given $l$,  $1\leq l \leq g$, consider  the sequence of polynomials
$(-1)^{\delta_l} \hat I^{[lp^s-1]}_{p^s}(u)$ and the series $\hat I^{(\ga_l)}(u)$. 
Here $u=(u_1,\dots,u_{n-1})$.

\vsk.2>
We multiply
the polynomials and the series  by the same factor $(u_1\cdots u_{n-2l})^{l}=(z_{n-2l}-z_n)^l$ and study the convergence
of the sequence of polynomials 
$J^{l,s}:= (u_1\cdots u_{n-2l})^{l}(-1)^{\delta_l} \hat I^{[lp^s-1]}_{p^s}(u)$
to the series $J^l:=(u_1\cdots u_{n-2l})^{l} \hat I^{(\ga_l)}(u)$. Introduce new variables\,:
\bea
&&
x_1= \prod_{i=1}^{n-2l}\prod_{j=1}^iu_j =\prod_{i=1}^{n-2l}(z_i-z_n)\,, 
\qquad
x_2=u_2\cdots u_{n-2l} =\frac{z_{n-2l}-z_n}{z_1-z_n}\,,
\\
&&
  x_3=u_3\cdots u_{n-2l}=\frac{z_{n-2l}-z_n}{z_2-z_n}, \quad
  \dots \quad
x_{n-2l}=u_{n-2l}=\frac{z_{n-2l}-z_n}{z_{n-2l-1}-z_n}\,,
\\
&&
x_{n-2l+1}=u_{n-2l+1}=\frac{z_{n-2l+1}-z_n}{z_{n-2l}-z_n},\qquad
x_{n-2l+2}=u_{n-2l+1}u_{n-2l+2} = \frac{z_{n-2l+2}-z_n}{z_{n-2l}-z_n},
\\
&&
\dots \quad
x_{n-1}=u_{n-2l+1}\cdots u_{n-1}=\frac{z_{n-1}-z_n}{z_{n-2l}-z_n}\,.
\eea
Let $x=(x_1,\dots,x_{n-1})$.  Then
\bea
J^l(x) 
=
x_1^{-\frac12}\,Q^l(x_2,\dots,x_{n-1})= x_1^{-\frac12} (Q^l_1(x_2,\dots,x_{n-1}),\dots,Q^l_n(x_2,\dots,x_{n-1})),
\eea
with coordinates $Q^l_j$ defined as follows.
If $j=1,\dots, n-1$, then
\bean
\label{J_j}
&&
\\
\notag
Q^l_j
&=&
 \,x_{j+1} \, {\sum}^{l,j} \,\binom{\frac{-3}2}{a_{j}}
\,\prod_{i=1, \,i\ne j}^{n-1} \binom{\frac{-1}2}{a_i}
\,
 \prod_{i=2}^{n-2l}x_{i}^{a_{i-1}}
\!\!\!\!\!  \prod_{i=n-2l+1}^{n-1} \!\!\! x_{i}^{a_i}\,,\quad j=1,\dots, n-2l-1,
\\
\notag
Q^l_j
&=&
 {\sum}^{l,j} \,\binom{\frac{-3}2}{a_{j}}
\,\prod_{i=1, \,i\ne j}^{n-1} \binom{\frac{-1}2}{a_i}
\,
 \prod_{i=2}^{n-2l}x_{i}^{a_{i-1}}
\!\!\!\!\!
  \prod_{i=n-2l+1}^{n-1}\!\!\!  x_{i}^{a_i}\,,\quad 
  \qquad
  j=n-2l,\dots, n-1,
 \eean
where the summation ${\sum}^{l,j}$
is over all $a_1,\dots,a_{n-1}\in\Z_{\geq 0}$
such that
$a_1+\dots+a_{n-2l}=a_{n-2l+1}+\dots+a_{n-1} + l-1$, if $j\leq n-2l$;
and such that
$a_1+\dots+a_{n-2l}=a_{n-2l+1}+\dots+a_{n-1} + l$, if $ n-2l<j\leq n-1$;
\bean
\label{Jn}
&&
\phantom{aaa}
Q^{l}_n = \, {\sum}^{l,n} \,\,\prod_{i=1}^{n-1}
\binom{\frac{-1}2}{a_i}
\,
 \prod_{i=2}^{n-2l}x_{i}^{a_{i-1}}
\!\!\!\!\!  \prod_{i=n-2l+1}^{n-1}\!\!\!  x_{i}^{a_i}\,,
\eean
where the summation ${\sum}^{l,n}$
is over all $a_1,\dots,a_{n-1}\in\Z_{\geq 0}$
such that 
$a_1+\dots+a_{n-2l}=a_{n-2l+1}+\dots+a_{n-1} + l$.

We also have
\bea
J^{l,s}(x) 
=
x_1^{\frac{p^s-1}2}\,Q^{l,s}(x_2,\dots,x_{n-2})= x_1^{\frac{p^s-1}2} (Q^{l,s}_1(x_2,\dots,x_{n-2}),
\dots,Q^{l,s}_n(x_2,\dots,x_{n-2})),
\eea
with coordinates $Q^{l,s}_j(x)$ defined as follows.
If $j=1,\dots, n-1$, then
\bean
\label{J_js}
&&
\\
\notag
Q^{l,s}_j 
&=&
 \,x_{j+1} \, {\sum}^{l,j} \,\binom{\frac{p^s-3}2}{a_{j}}
\,\prod_{i=1, \,i\ne j}^{n-1} \binom{\frac{p^s-1}2}{a_i}
\,
 \prod_{i=2}^{n-2l}x_{i}^{a_{i-1}}
\!\!\!\!\!  \prod_{i=n-2l+1}^{n-1} \!\!\! 
x_{i}^{a_i}\,,\quad j=1,\dots, n-2l-1,
\\
\notag
Q^{l,s}_j
&=&
 {\sum}^{l,j} \,\binom{\frac{p^s-3}2}{a_{j}}
\,\prod_{i=1, \,i\ne j}^{n-1} \binom{\frac{p^s-1}2}{a_i}
\,
 \prod_{i=2}^{n-2l}x_{i}^{a_{i-1}}
\!\!\!\!\!  \prod_{i=n-2l+1}^{n-1}\!\!\!  x_{i}^{a_i}\,,\quad 
\qquad
j=n-2l,\dots, n-1,
 \eean
where the summation ${\sum}^{l,s,j}$
is over all $a_1,\dots,a_{n-1}\in\Z$,  $0\leq a_i\leq \frac{p^s-1}2$,
such that
$a_1+\dots+a_{n-2l}=a_{n-2l+1}+\dots+a_{n-1} + l-1$, if $j\leq n-2l$;
and such that
$a_1+\dots+a_{n-2l}=a_{n-2l+1}+\dots+a_{n-1} + l$, if $ n-2l<j\leq n-1$;
\bean
\label{Jns}
&&
\phantom{aaa}
Q^{l,s}_n = \, {\sum}^{l,n} \,\,\prod_{i=1}^{n-1}
\binom{\frac{p^s-1}2}{a_i}
\,
 \prod_{i=2}^{n-2l}x_{i}^{a_{i-1}}
\!\!\!\!\!   \prod_{i=n-2l+1}^{n-1} \!\!\!  x_{i}^{a_i}\,,
\eean
where the summation ${\sum}^{l,s,n}$
is over all $a_1,\dots,a_{n-1}\in\Z$,  $0\leq a_i\leq \frac{p^s-1}2$,
such that 
$a_1+\dots+a_{n-2l}=a_{n-2l+1}+\dots+a_{n-1} + l$.

\begin{prop}
\label{prop Q conv}

The power series $Q^l(x_2,\dots,x_{n-1})$ uniformly convergens on $D_{0,1}^{n-2}$.
The sequence of polynomial functions 
$(Q^{l,s}(x_2,\dots,x_{n-1}))_{s=1}^\infty$ uniformly converges on 
$D_{0,1}^{n-2}$ to the function $Q^l(x_2,\dots,x_{n-1})$.

\end{prop}

\begin{proof}
The fact that the binomials 
 $\binom{-\frac{1}2}{a}, \binom{-\frac{3}2}{a}$ are $p$-adic integers
  implies the uniform convergence of the power series $Q^l(x_2,\dots,x_{n-1})$ on $D_{0,1}^{n-2}$.
The proof of the uniform convergence of
$(Q^{l,s}(x_2,\dots,x_{n-1}))_{s=1}^\infty$\,to\, $Q^l(x_2,\dots,x_{n-1})$
follows from Lemma \ref{lem BIN} in the same way as the uniform convergence in the proof of Proposition \ref{prop convP g=1}.
\end{proof}

Consider the formal series $J^l(x)=x_1^{-\frac12}\,Q^l(x_2,\dots,x_{n-1})$
and the sequence of polynomials 
$J^{l,s}(x) 
=
x_1^{\frac{p^s-1}2}\,Q^{l,s}(x_2,\dots,x_{n-2})$
as functions on $D_{\al,1}\times D_{0,1}^{n-2}$, where
$\al = \beta^2$ for some $\beta\in\F_p$. Then the function $x_1^{1/2} : D_{\al,1}\to D_{\beta,1}$ is well-defined
and the series $J^l(x)$ is a well-defined function on  $D_{\al,1}\times D_{0,1}^{n-2}$.

\begin{thm}
\label{thm last}

The sequence of polynomial functions 
$\big(J^{l,s}(x)\big)_{s=1}^\infty$ uniformly converges on 
$D_{\al,1}\times D_{0,1}^{n-2}$ to the function 
$\om(\beta) J^{l}(x)$\,.

\end{thm}

\begin{proof}
The theorem follows from Lemma \ref{lem x=y2} and Proposition \ref{prop Q conv}.
\end{proof}

\appendix

\section{ { {The case $n=3$ and Dwork's theory}}}
\label{appendix}

\begin{center}
{ by  Steven Sperber
and Alexander Varchenko}
\end{center}

\bigskip

In this appendix we  consider only the special case  $n=3$ of  previous considerations and
 show how this special case is related to Dwork's theory
in the classical paper \cite{Dw}.

\subsection{Dwork on Legendre family}
\label{sec A1}

\subsubsection{} 

Consider the   Legendre family of elliptic curves  $E(\la)$ defined by the affine equation 
\bea
y^2=x(x-1)(x-\la).
\eea
 Let $\ga=\ga(\la)$ be a family of 1-cycles on
the curves $E(\la)$
 flat under the Gauss-Manin connection. Then the function  
\bean
\label{Kg}
{h}^{(\ga)}(\la) = \int_{\ga} 
\frac {dx}{y}
\eean
satisfies the hypergeometric differential equation
\bean
\label{HEa}
\la(1-\la) h'' +(1-2\la)h'-(1/4)h=0.
\eean
All solutions of this equation are obtained in this way.
One of the solutions $h^{(\delta_1)}(\la)$, for a suitable $\delta_1$, equals 
\bean
\label{Fps}
F(\la) = {}_2F_1\Big(\frac12,\frac12;1;\la\Big)=\sum_{k = 0}^\infty\binom{-1/2}{k}^2\la^k\,,
\eean
where ${}_2F_1$ is the classical hypergeometric function.

\subsubsection{} If $h(\la)$ is the elliptic integral in  \eqref{Kg}, then its derivative is 
\bean
\label{dKg}
h' = \frac12 \int_{\ga} \frac{dx}{(x-\la)  y}.
\eean

\subsubsection{}
Equation $\eqref{HEa}$ can be written as a system of first order linear differential equations for column 2-vectors
$I=(h, h')$,
\bean
\label{DDE}
\frac{d I}{d\la} = B(\la)I, \qquad 
B(\la)=
\begin{pmatrix}
 0 &    1
 \\
\frac{1}{4\la(1-\la)}   & \frac{2\la-1}{\la(1-\la)}
\end{pmatrix} .
\eean

\subsubsection{}

In \cite{Dw}
Dwork considers equation \eqref{HEa} over the field $\Q_p$ of 
$p$-adic numbers and studies analytic properties of the solution
$F(\la)$. We remind these properties below.

\subsubsection{}
Let $D\subset \Q_p$ be an open subset. A function  $f:D\to \Q_p$ is called analytic 
at $\al\in D$ if $f$ can be presented by a power series 
$\sum_{k=0}^\infty c_k(\la-\al)^k$ with nonzero radius of convergence.

\vsk.2>
A function  $f:D\to \Q_p$ is called analytic 
on $D$ if $f$ is the uniform limit on $D$ of a sequence of rational function
regular on $D$. In that case $f$ is analytic at every point of $D$.

\vsk.2>

For $i=1,\dots,N$ let $f_i:D_i\to\Q_p$ be an analytic function on some domain $D_i$.
Assume that $D_i\cap D_{i+1}\ne\emptyset$ for $i=1,\dots,N-1$,
and $f_i=f_{i+1}$ on $D_i\cap D_{i+1}$, then we say that the collection of functions 
$f_i$ defines an analytic element on $\cup_{i=1}^N D_i$. Cf. \cite[Section 0]{Dw}.

\subsubsection{}
Igusa noted in \cite{Ig} that modulo $p$ the polynomial
\bean
\label{Igu} 
g(\la) = \sum_{j=0}^{(p-1)/2} \binom{-1/2}{j}^2\la^j 
\eean
is the unique polynomial solution of equation \eqref{HEa} of degree less than $p$
up to multiplication by a constant.
Define 
\bea
\frak D_1 = \{\la\in \Z_p\ |\ |g(\la)|_p=1\},
\qquad
\frak D_2 = \{\la\ |\ \la^{-1} \in \frak D_1 \},
\qquad
\frak D= \frak D_1\cup\frak D_2.
\eea
Notice that $\frak D_1$, $\frak D_2$ are open and $\frak D_1\cap \frak D_2\ne\emptyset$. More precisely,
\bea
\fD_1\cap \fD_2 = \{\la\in\Z_p\ |\ |g(\la)|_p=1, \, |\la|_p = 1\}.
\eea

\subsubsection{} 
\label{sec A17}

Dwork considers the functions
\bean
\label{def etf}
f(\la) = \frac{F(\la)}{F(\la^p)},
\qquad
\eta(\la)=\frac{F'(\la)}{F(\la)}, 
\eean
defined in a neighborhood 
of $0\in\frak D_1$ as ratios of the corresponding convergent power series expansions. 

\vsk.2>
Dwork proves that $f(\la)$ can be analytically continued to the domain $\frak D_1$.
For that, he indicates a sequence of regular rational functions
 on $\frak D_1$, that sequence uniformly converges  on $\frak D_1$,
and its limit  equals $F(\la)/F(\la^p)$  in a neighborhood of $0$,
see \cite[Lemma 3.4]{Dw}.

\smallskip
From that Dwork deduces that $\eta(\la)$ has analytic continuation to the domain
$\frak D_1$ in the same sense, see  \cite[Lemma 3.1]{Dw}.

\vsk.2>
Since $\eta(\la)$ is analytic on $\frak D_1$, the function $\eta(1/\la)$ is analytic on  $\frak D_2$.

\vsk.2>
Using the properties of equation \eqref{HEa} Dwork shows that
$\eta(1-\la) = -\eta(\la)$ on $\frak D_1$ and shows that 
$\eta (\la) = -\eta(1/\la)/\la^2 -1/(2\la)$ on $\frak D_1\cap \frak D_2$.
Hence the function $\eta(\la)$ on $\frak D_1$ and the function
 $-\eta(1/\la)/\la^2 -1/(2\la)$ on $\frak D_2$
define an analytic element on $\frak D$. 

\vsk.2>
We will use the  formulas
\bean
\label{DD}
\eta(1-\la) = -\eta(\la),\qquad
\eta(1/\la) = -\la^2\eta(\la)-\la/2\,,
\eean
in Section \ref{sec A5}.

\subsubsection{} 
\label{sec A18}

For $\al\in \Q_p$ let $V_\al$ be the space of germs at $\al$ of holomorphic solutions of equation
\eqref{HEa}. For $\al\ne 0,1$ we have $\dim V_\al=2$  and for $\al=0,1$ we have
$\dim V_\al=1$.

\vsk.2>

For $\al\in\fD_1$ let $U_\al$ be the space of germs at $\al$ 
of analytic functions defined by the equation
\bean
\label{3.9}
\frac{du}{d\la} = \eta(\la) u\,.
\eean
By \cite[Lemma 3.2]{Dw}, \, $U_\al$ is a subspace of $V_\al$. 
We have  $U_\al=V_\al$ for $\al=0,1$.

For $\al\in\fD_2$ let $U_\al'$ be the space of germs at $\al$ 
of analytic functions defined by the equation
\bean
\label{4.9}
\frac{du}{d\la} = (-\eta(1/\la)/\la^2 -1/(2\la)) u\,.
\eean
By \cite[Lemma 3.2]{Dw}, \, $U_\al'$ is a subspace of $V_\al$. 
For $\al\in \fD_1\cap\fD_2$ we have $U_\al=U_\al'$.

\subsubsection{}
\label{sec A19}

By \cite[Lemma 4.2]{Dw}, for $\al\in \fD$\, the subspace $U_\al\subset V_\al$ also
can be characterized 
as the subspace of germs at $\al$ of holomorphic functions bounded in their disc of convergence.

\vsk.2>
More precisely, let $\C_p$ be the metric completion of the algebraic closure $\bar \Q_p$ of the field 
$\Q_p$. Let $\al\in \fD$. Let $u(\la) = \sum_{k=0}^\infty c_k(\la-\al)^k$ be an element of $V_\al$. 
Consider $u(\la)$ as a germ at $\al \in \Q_p\subset \C_p$ of an analytic function on $\C_p$.
The germ $u(\la)$ is called bounded on its disc of convergence
 if $u(\la)$ is bounded on its disc of convergence in $\C_p$.

Let $r$ be the radius of convergence of   $u(\la)=\sum_{k=0}^\infty c_k(\la-\al)^k$.
Define  $|u(\la)|_0  = \sup_k |c_k| r^k$. Then $u(\la)$
 is bounded in its disc of convergence 
 if and only if $|u(\la)|_0 < \infty$.

\subsubsection{} 
\label{sec A110}

The function $F(\la)$ is a holomorphic solution of equation \eqref{HEa} on the disc $D_{0,1}$.
A second solution of \eqref{DDE} is of the form $G(\la) = F(\la) \log \la + H(\la)$, where the function
$H(\la)$ is holomorphic on $D_{0,1}$. Dwork specifies $H(\la)$ by \cite[Equation (4.19)]{Dw}. Then
\bean
\label{U la}
\mc F(\la) = 
\begin{pmatrix}
 F &    G
 \\
 F'   & G' 
\end{pmatrix} 
\eean
is a fundamental matrix of solutions of equation \eqref{DDE}.

\vsk.2>
In \cite{Dw} Dwork introduces a $2\times 2$-matrix function $\mc A(\la)$, 
and then proves the formula
\bean
\label{mc A}
\mc A(\la) = \mc F(\la)\, M \, \mc F(\la^p)^{-1}, \qquad M 
= 
\begin{pmatrix}
 (-1)^{(p-1)/2} &  b
 \\
 0  & (-1)^{(p-1)/2} p
\end{pmatrix},
\eean
where $b$ is a suitable  number, see  \cite[Lemma 6.2]{Dw} and formulas on page 72 in \cite{Dw}.
Dwork shows that $\mc A(\la)$  extends to an analytic function on the domain $\fD_3\cup\fD_4$,
where
\bean
\label{D34}
&
\fD_3 = \{\lambda \in \mathbb{Z}_p \ | \ |\lambda|_p =1, |\lambda - 1|_p= 1\}, 
\qquad
\fD_4 = \{\la\in \Z_p\ |\  \eps <|\la|_p < 1\},
\\
\notag
& \fD_3 \cup \fD_4 =  
 \{\lambda \in \mathbb{Z}_p\ |\  \eps < |\lambda|_p  \leq 1,\, |\lambda - 1|_p = 1\}.
\eean
Here $\eps$ is some explicit number, $0<\eps < 1$.
See the bottom of page 62 in \cite{Dw} and the first sentence of the proof of
 Theorem 6 in \cite{Dw}.

\vsk.2>
Formula \eqref{mc A} immediately implies that
\bean
\label{Frob}
\mc A(\la) \mc F(\la^p) = \mc F(\la) M\,
\eean
on $\fD_4$. The matrix $\mc A(\la)$ is called the matrix of the Frobenius transformation
of solutions of equation \eqref{DDE} 
relative to the fundamental matrix $\mc F(\la)$.

\vsk.2>

It follows from formula \eqref{Frob} that 
\bean
\label{AF}
\mc A(\la) 
\begin{pmatrix}
F(\la^p)
 \\
 F'(\la^p)
 \end{pmatrix}=
 (-1)^{(p-1)/2} 
 \begin{pmatrix}
F(\la)
 \\
 F'(\la)
 \end{pmatrix}
 \eean
 on $\fD_4$. This  can be reformulated as  the relation
\bean
\label{FRE}
\mc A(\la) 
\begin{pmatrix}
1
 \\
 \eta(\la^p) 
 \end{pmatrix}=
 (-1)^{(p-1)/2} f(\la) 
 \begin{pmatrix}
1
 \\
 \eta(\la)
 \end{pmatrix}
 \eean
 on $\fD_4$.  By the already formulated analytic properties of $\eta(\la)$ and $\mc A(\la)$,\,
  relation \eqref{FRE}
 can be analytically continued to the domain $\fD_1\cap \fD_3$.

\vsk.2>

Equation \eqref{FRE} implies that for  any $\al \in \F^\times_p-\{1\}$
such that $\om(\al) \in \fD_1$, the vector $(1,\eta(\om(\al)))$ is an eigenvector of the Frobenius matrix
$\mc A(\om(\al))$ with eigenvalue 
\linebreak
$(-1)^{(p-1)/2} f(\om(\al))$,
\bean
\label{eig}
\mc A(\la) 
\begin{pmatrix}
1
 \\
 \eta(\om(\al)) 
 \end{pmatrix}=
 (-1)^{(p-1)/2} f(\om(\al)) 
 \begin{pmatrix}
1
 \\
 \eta(\om(\al))
 \end{pmatrix}.
 \eean

It is known that 
the zeta function of  the elliptic curve defined over
$\F_p$  by the equation
$y^2=x(x-1)(x-\al)$ has two zeros, which are
$1/((-1)^{(p-1)/2} f(\om(\al)))$, $(-1)^{(p-1)/2} f(\om(\al))/p$.
It is also known that $|f(\om(\al))|_p=1$.
The number $(-1)^{(p-1)/2} f(\om(\al))$ is 
called the unit root.  
See \cite{Dw} and also \eqref{def z}.

\subsection{KZ equations}
The KZ equations \eqref{KZ} for $n=3$  is the following system of differential and algebraic  equations
for a column $3$-vector $I=(I_1,I_2,I_3)$ depending on variables $z=(z_1,z_2,z_3)$\,:
\bean
\label{KZZ}
\frac{\partial I}{\partial z_1}  
&=&
   {\frac 12} \Big(
   \frac{\Omega_{12}}{z_1 - z_2} 
+\frac{\Omega_{13}}{z_1 - z_3} 
      \Big) I ,
\qquad
\frac{\partial I}{\partial z_2}  
= 
   {\frac 12} \Big(
   \frac{\Omega_{21}}{z_2 - z_1} 
+\frac{\Omega_{23}}{z_2 - z_3} 
      \Big) I ,
      \\
\notag
\frac{\partial I}{\partial z_3}  
&=& 
   {\frac 12} \Big(
   \frac{\Omega_{31}}{z_3 - z_1} 
+\frac{\Omega_{32}}{z_3 - z_2} 
\Big) I,
\qquad \quad
0= I_1+I_2+I_3,
\eean 
where $\Om_{ij}=\Om_{ji}$ and
\bea
 \Omega_{12} =  
 \begin{pmatrix}
 -1 & 1 & 0 
 \\
 1   & -1 & 0
 \\
0   & 0 & 0
\end{pmatrix} ,
\quad
 \Omega_{13} =  
 \begin{pmatrix}
 -1 & 0 & 1 
 \\
 0   & 0 & 0
 \\
1   & 0 & -1
\end{pmatrix} ,
\quad
 \Omega_{23} =  
 \begin{pmatrix}
 0 & 0 & 0 
 \\
 0   & -1 & 1
 \\
0   & 1 & -1
\end{pmatrix} .
\eea       
We introduce new variables
\bean
\label{id co}
u_1=z_1-z_3, \qquad
u_2=\frac{z_2-z_3}{z_1-z_3},
\qquad
u_3=z_1+z_2+z_3,
\eean
see \eqref{uz}.
Then system \eqref{KZZ} takes the form
\bean
\label{KZZu}
\frac{\partial I}{\partial u_1}  
&=&
   {\frac 12} 
   \frac{\Omega_{12}+\Omega_{13}+\Omega_{23}}{u_1} I\,,
\qquad
\frac{\partial I}{\partial u_2}  
= 
   {\frac 12} \Big(
   \frac{\Omega_{12}}{u_2-1} + \
   \frac{\Omega_{23}}{u_2} \Big) I,
\\
\notag
\frac{\partial I}{\partial u_3}  
&=& 0,
   \qquad
   \phantom{aaaaaaaaaaaaaaaaaa}
   0= I_1+I_2+I_3.
\eean 
The variables in system \eqref{KZZu} are separated, cf. \eqref{Ku}.

Denote 
$\tilde W=\{ (I_1,I_2,I_3)\ | \  I_1+I_2+I_3 =0\}$. Then
\bean
\label{Om 1}
(\Omega_{12}+\Omega_{13}+\Omega_{23})|_{\tilde W} = -3 \on{Id}\,.
\eean
Hence all solutions of system \eqref{KZZu} have the form
\bean
I = u_1^{-3/2} (J_1(u_2), J_2(u_2), J_3(u_2)), \qquad
J_1+J_2+J_3 =0,
\eean
where the column vector $J(u_2)$ is a solution of the differential equation
\bean
\label{J eqn}
\frac{\partial J}{\partial u_2}  
= 
   {\frac 12} \Big(
   \frac{\Omega_{12}}{u_2-1} + \
   \frac{\Omega_{23}}{u_2} \Big) J\,.
   \eean

\subsection{Solutions over $\C$}
\label{sec11.4}

Any solution of system \eqref{KZZ} has the form
\bean
\label{I3}
I^{(\ga)}(z) =
\int_\ga
\Big(\frac {1}{x-z_1}, \frac {1}{x-z_2},
\frac {1}{x-z_3}\Big) 
\frac {dx}{\sqrt{(x-z_1)(x-z_2)(x-z_3)}}\,.
\eean
where  $\ga$ is a flat family of 1-cycles on the elliptic curves of our family of curves. 

We change $x$ and $z$
in this integral by setting  $x=(z_1-z_3)w+z_3$ and $z=z(u)$ as in \eqref{id co}.
Then integral \eqref{I3} takes the form
\bean
\label{hat I3}
I^{(\ga)}(u_1,u_2) 
=
u_1^{-3/2}
\int_\ga
\Big(\frac {1}{w-1}, \frac {1}{w-u_2},
\frac {1}{w}\Big) \frac {dw}{\sqrt{(w-1)(w-u_2)\,w}}\,.
\eean

We take $\ga=\ga_1$ to be the circle $|w|=1/2$
oriented counter-clockwise.
We  assume that $u_2$ lies in this circle.
We fix the branch of $\sqrt{(w-1)(w-u_2)\,w}$
over the circle by choosing the argument of 
$\sqrt{(w-1)(w-u_2)\,w}$ at $w=1/2$, $u_2=0$  to be $\pi/2$.
We multiply the circle with the chosen branch of the integrand
by $-\frac {1}{2\pi}$. This finishes the description of $\ga_1$.
See the definition of cycles $\ga_l$ in Section \ref{sec 9.5}.

\vsk.2> 
We
expand  the integral $I^{(\ga_1)}(u_1,u_2)$ 
 as a power series in $u_2$ and obtain
\bean
\label{hat I33}
 I^{(\ga_1)}(u_1, u_2) 
&=&
 u_1^{-3/2}  \sum_{a=0}^\infty \Big(\binom{-\frac32}{a}\binom{-\frac12}{a},
\binom{-\frac12}{a+1}\binom{-\frac32}{a},
\binom{-\frac12}{a+1}\binom{-\frac12}{a}\Big) \,u_2^a
\\
\notag
&=&
 u_1^{-3/2}  \sum_{a=0}^\infty \binom{-\frac12}{a+1}\binom{-\frac32}{a}
\Big(\frac{a+1}{-1/2-a},1,\frac{-1/2}{-1/2-a}\Big) u_2^a\, ,
\eean
see Theorem \eqref{thm hot I}.   Denote
\bean
\label{hI}
&
I
:= I^{(\ga_1)} =
 (z_1-z_3)^{-3/2}  \sum_{a=0}^\infty \binom{-\frac12}{a+1}\binom{-\frac32}{a}
\Big(\frac{a+1}{-1/2-a},1,\frac{-1/2}{-1/2-a}\Big) \Big(\frac{z_2-z_3}{z_1-z_3}\Big)^a\,
\\
\label{hIu}
&
=\  \ \  u_1^{-3/2}  \sum_{a=0}^\infty \binom{-\frac12}{a+1}\binom{-\frac32}{a}
\Big(\frac{a+1}{-1/2-a},1,\frac{-1/2}{-1/2-a}\Big) u_2^a\, .
\eean
This series is a solution of  system \eqref{KZZ}.

\begin{rem}
Formulas \eqref{hat I3} and \eqref{hat I33} imply that
\bean
\label{hat I3333}
&
I(u_1, u_2) 
=
 u_1^{-3/2} 
 \Big({}_2F_1\Big(\frac12,\frac32;1;u_2\Big),\,-\frac12\, {}_2F_1\Big(\frac32,\frac32;2;u_2\Big),
 -\frac12\, {}_2F_1\Big(\frac12,\frac32;2;u_2\Big) \Big)\,,
 \eean
 where 
${}_2F_1(a,b;c;\la)$ is the classical hypergeometric function.

\end{rem}

\subsection{Solutions as vectors of first derivatives}
\label{sec11.5}

Introduce the function
\bean
\label{Lg}
\ell^{(\ga)}(z) =
\int_\ga
\frac {dx}{\sqrt{(x-z_1)(x-z_2)(x-z_3)}}\,.
\eean
Then
\bean
\label{KI}
I^{(\ga)}(z) 
=
\,
2\,
\Big(\frac {\der \ell^{(\ga)}}{\der z_1},
\frac {\der \ell^{(\ga)}}{\der z_2},
\frac {\der \ell^{(\ga)}}{\der z_3}\Big).
\eean
Changing the variable $x=w(z_1-z_3)+z_3$ we write
\bean
\label{LK}
\ell^{(\ga)}(z) 
&=&
(z_1-z_3)^{-1/2} \int_\ga \frac {dw}{\sqrt{(w-1)(w-\frac{z_2-z_3}{z_1-z_3})w}}\,,
\\
\notag
& =&
 (z_1-z_3)^{-1/2} h^{(\ga)}\Big(\frac{z_2-z_3}{z_1-z_3}\Big)
 \\
 \notag
 & =& u_1^{-1/2} h^{(\ga)}(u_2),
\eean
where $h^{(\ga)}(\la)$ is the elliptic integral in \eqref{Kg} and 
$h^{(\ga)}(\la)$  is a solution of equation \eqref{HEa}.
Denote $h(\la):=h^{(\ga)}(\la)$. Then
\bean
\label{Lh}
&
\\
\notag
&
\!\!
I^{(\ga)}
\!=\!
(z_1-z_3)^{-3/2}\big( \!\!- \!h\big(\frac{z_2-z_3}{z_1-z_3}\big)\! -\!2h'\big(\frac{z_2-z_3}{z_1-z_3}\big)
\frac{z_2-z_3}{z_1-z_3},
2h'\big(\frac{z_2-z_3}{z_1-z_3}\big),
 h\big(\frac{z_2-z_3}{z_1-z_3}\big) \!+\!2h'\big(\frac{z_2-z_3}{z_1-z_3}\big) \frac{z_2-z_1}{z_1-z_3}\big)
\\
\label{Lhu}
&
=
u_1^{-3/2}( - h(u_2) -2h'(u_2)u_2,\,
2h'(u_2),\,
 h(u_2) +2h'(u_2)(u_2-1)).
\eean
Formula \eqref{Lhu} relates
 solutions of system  \eqref{DDE} and solutions of system
\eqref{KZZu}. If $(h,h')$ is a solution of system \eqref{DDE}, then equation
\bean
\label{H-KZ}
\begin{pmatrix}
 I_1(u_1,u_2)
 \\
 I_2(u_1,u_2)
\\
 I_3(u_1,u_2)
 \end{pmatrix} 
=
u_1^{-3/2}
\begin{pmatrix}
 -1 & -2u_2
 \\
 0   &   2
 \\
1 & 2u_2-2
\end{pmatrix} 
\begin{pmatrix}
 h(u_2)
 \\
 h'(u_2)
 \end{pmatrix} 
\eean
gives  a solution $(I_1, I_2, I_3)$ of system \eqref{KZZu}. Conversely,
if $(I_1, I_2, I_3)$ is a solution of system \eqref{KZZu} then formula
\bean
\label{KZ-H}
\begin{pmatrix}
 h(u_2)
 \\
 h'(u_2) 
\end{pmatrix} 
=
u_1^{3/2}
\begin{pmatrix}
 -1 & -u_2 & 0
 \\
 0   &   1/2 & 0
\end{pmatrix} 
\begin{pmatrix}
 I_1(u_1,u_2)
 \\
 I_2(u_1,u_2)
 \\
 I_3(u_1,u_2)
 \end{pmatrix} 
\eean
gives a solution $(h,h')$ of system \eqref{DDE}.

Using the cycle $\ga_1$ we can evaluate
\bean
\label{Kz}
\ell^{(\ga_1)} &=& (z_1-z_3)^{-1/2} F\Big(\frac{z_2-z_3}{z_1-z_3}\Big)
\\
\notag
 &=& u_1^{-1/2} F(u_2),
\qquad\quad\quad\quad\quad
\text{where}\ \  \
F(\la) ={}_2F_1\Big(\frac12,\frac12;1;\la\Big).
\eean
Denote 
\bean
\label{mc K}
\ell:=\ell^{(\ga_1)}, 
\qquad
\mc I \, :=\,\frac I\ell =
\frac2\ell\,\Big(\frac{\der \ell}{\der z_1},\frac{\der \ell}{\der z_2}, \frac{\der \ell}{\der z_3}\Big),
\eean
where  $I$ is defined in \eqref{hI}.
Formulas \eqref{Lh} and \eqref{Lhu} imply
\bean
\label{K/Kz}
&
\mc I = \frac 1{z_1-z_3}\big(\! -1 -2\eta\big(\frac{z_2-z_3}{z_1-z_3}\big)\frac{z_2-z_3}{z_1-z_3},\,
2\eta\big(\frac{z_2-z_3}{z_1-z_3}\big),\,
1+2\eta\big(\frac{z_2-z_3}{z_1-z_3}\big) \frac{z_2-z_1}{z_1-z_3}\big),
\\
\label{K/K}
&\phantom{aaa}
 = \frac1{u_1}
\big(\!- 1 - 2\eta(u_2) u_2,\, 2\eta(u_2),\, 1 + 2\eta(u_2)(u_2-1)\big)\,,
\eean
where the function $\eta(\la)$ is defined in \eqref{def etf}.

\subsection{Six coordinate systems} 
\label{sec A5}

System \eqref{KZZ} of KZ equations has 6 distinguished coordinate systems (asymptotic zones).
They are labeled by permutations $\si=(i,j,k)\in S_3$.  
The coordinate system $u^\si=(u_1^\si,u_2^\si, u_3^\si)$ is defined by the formulas
\bean
\label{co si}
u_1^{\si}=z_{i}-z_{k}, 
\qquad
u_2^{\si}=\frac{z_{j}-z_{k}}{z_{i}-z_{k}}, \qquad
u_3^\si=z_1+z_2+z_3.
\eean
For the  identity element $\id =(1,2,3)$ the corresponding
coordinate system is 
defined in \eqref{id co}.

\vsk.2>
Having one of these coordinate systems we
repeat the constructions of Sections \ref{sec11.4}-\ref{sec11.5} and construct a 
scalar function $\ell^\si(z)$ and vector-valued functions
$I^\si(z)$,  $\mc I^\si(z)$, such that $\mc I^\si(z) =  I^\si(z)/\ell^\si(z)$.
For the  identity element $\id =(1,2,3)$ these  functions are $\ell(z)$, $I(z)$, $\mc I(z)$
 in  \eqref{Kz}, \eqref{hI}, \eqref{K/Kz}. Notice that the functions $\ell(z)$ and $I(z)$
are defined as integrals over $\ga_1$, and that $\ga_1$ is defined 
 with the help of  coordinates $u_1,u_2, u_3$.  
 
 \vsk.2>
 For any $\si$ the function $I^\si(z)$ is a power series solution of system \eqref{KZZ} in the chart with coordinates
 $u^\si$, see \eqref{hat I33} and \eqref{hI}.

\vsk.2>
Below we list the functions $\mc I^\si $\,:
 \bean
\label{Ksi}
&
\\
\notag
&
\mc I^{123} = \frac1{z_1-z_3}\big( \!-1 -2\eta\big(\frac{z_2-z_3}{z_1-z_3}\big)\frac{z_ 2-z_3}{z_1-z_3},\,
2\eta\big(\frac{z_2-z_3}{z_1-z_3}\big),\,
1  +\eta\big(\frac{z_2-z_3}{z_1-z_3}\big) \frac{z_2-z_1}{z_1-z_3}\big),
\\
\notag
&
\mc I^{321} = \frac1{z_3-z_1}\big(1 +2\eta\big(\frac{z_2-z_1}{z_3-z_1}\big) \frac{z_2-z_3}{z_3-z_1},\, 
2\eta\big(\frac{z_2-z_1}{z_3-z_1}\big),\,
  -1 -2\eta\big(\frac{z_2-z_1}{z_3-z_1}\big)\frac{z_ 2-z_1}{z_3-z_1}\big),
\\
\notag
&
\mc I^{213} = \frac1{z_2-z_3}\big( 
2\eta\big(\frac{z_1-z_3}{z_2-z_3}\big),\,
-1 - 2\eta\big(\frac{z_1-z_3}{z_2-z_3}\big)\frac{z_1-z_3}{z_2-z_3},\,
 1 +2\eta\big(\frac{z_1-z_3}{z_2-z_3}\big) \frac{z_1-z_2}{z_2-z_3}\big),
\\
\notag
&
\mc I^{132} = \frac1{z_1-z_2}
\big( -1 -\eta\big(\frac{z_3-z_2}{z_1-z_2}\big)\frac{z_ 3-z_2}{z_1-z_2},\,
 1 +2\eta\big(\frac{z_3-z_2}{z_1-z_2}\big) \frac{z_3-z_1}{z_1-z_2},\,
 2\eta\big(\frac{z_3-z_2}{z_1-z_2}\big)\big),
\\
\notag
&
\mc I^{231} = \frac1{z_2-z_1} \big( 
1 +2\eta\big(\frac{z_3-z_1}{z_2-z_1}\big) \frac{z_3-z_2}{z_2-z_1},\,
-1 -2\eta\big(\frac{z_3-z_1}{z_2-z_1}\big)\frac{z_ 3-z_1}{z_2-z_1},\,
2\eta\big(\frac{z_3-z_1}{z_2-z_1}\big) \big),
\\
\notag
&
\mc I^{312} = \frac{1}{z_3-z_2}\big( 2\eta\big(\frac{z_1-z_2}{z_3-z_2}\big),\,
 1 +2\eta\big(\frac{z_1-z_2}{z_3-z_2}\big) \frac{z_1-z_3}{z_3-z_2},\,
 -1 -2\eta\big(\frac{z_1-z_2}{z_3-z_2}\big)\frac{z_ 1-z_2}{z_3-z_2)^2}\big).
\eean

\begin{thm}
\label{thm ael}
For $(i,j,k)\in S_3$ consider the three functions
$\mc I^{ijk}$, $\mc I^{kji}$, $\mc I^{jik}$.
Then $\mc I^{kji}$ is transformed to $\mc I^{ijk}$ by application of formula
$\eta(1-\la) = -\eta(\la)$ and  $\mc I^{jik}$ is transformed to  $\mc I^{ijk}$
by application of formula $\eta(1/\la) = -\la^2\eta(\la)-\la/2$.

\end{thm}

\begin{proof}
The proof is straightforward. For example we check the statement for $(i,j,k)=(1,2,3)$.
In this case the functions $\mc I^{123}$, $\mc I^{321}$, $\mc I^{132}$ are
\bea
&
\frac 1{u_1}\big(
-1 - 2\eta(u_2)u_2,\,
2\eta(u_2) ,\,
1 + 2\eta(u_2) (u_2-1)\big)\,,
\\
&
\frac1{u_1}\big(
-1 + 2\eta(1-u_2)u_2, \, 
-2\eta(1-u_2), \,
1 + 2\eta(1-u_2)(1-u_2)\big)\,,
\\
&
\frac{1}{u_1u_2}
\big( 2\eta\big(\frac1{u_2}\big) ,\,
-1 - 2\eta\big(\frac1{u_2}\big) \frac{1}{u_2},\,
1 + 2\eta\big(\frac1{u_2}\big)\frac{1-u_2}{u_2}\big)\,,
\eea
where $u_1,u_2$ are defined  in \eqref{id co}. Then formula 
$\eta(1-u_2) = -\eta(u_2)$ transforms the second function to the first and
the formula
$\eta(1/u_2) = -u_2^2\eta(u_2)-u_2/2$ transforms the third function to the first.
\end{proof}

Define 
\bea
\tilde {\frak D}_0 = \{(z_1,z_2,z_3)\in \Q_p^3\ |\ z_i\ne z_j\,\ \forall i\ne j\}.
\eea
For any $\si=(i,j,k)\in S_3$ define
\bea
&&
\phantom{aaaaa}
\tilde {\frak D}_1^\si =\Big\{(z_1,z_2,z_3)\in \tilde {\frak D}_0\  \Big|\ \
\frac{z_j-z_k}{z_i-z_k}\in \Z_p,\ \Big| g\Big(\frac{z_j-z_k}{z_i-z_k}\Big)\Big|_p = 1\Big\},
\\
&&
\tilde{\frak D}_2^\si = \Big\{(z_1,z_2,z_3)\in \tilde {\frak D}_0\  \Big|\ 
\frac{z_i-z_k}{z_j-z_k} \in \tilde {\frak D}_1^\si\Big\},
\qquad
\tilde{\frak D}^\si= \tilde{\frak D}_1^\si\cup\tilde {\frak D}_2^\si,
\qquad
\tilde{\frak D} = \sum_{\si\in S_3}\tilde{\frak D}^\si,
\eea
where the function $g$ is defined in \eqref{Igu}.

For any any $(i,j,k)\in S_3$ the functions 
$\mc I^{ijk}$, $\mc I^{kji}$, $\mc I^{jik}$  define an analytic element on
$\tilde{\frak D}^\si$, see Section \ref{sec A17} and \cite{Dw}.
Theorem \ref{thm ael} implies the following corollary.

\begin{cor}
\label{cor A1} The functions $(\mc I^{ijk})_{(i,j,k)\in S_3}$ define an analytic element on
$\tilde{\frak D}$.
\qed
\end{cor}

\begin{rem}
Dwork's formulas \eqref{DD} present the $S_3$-symmetries of the analytic element 
$(\eta(\la)$, $-\eta(1-\la)$, $-\eta(1/\la)/\la^2 -1/(2\la))$.
Dwork's $S_3$-symmetries reformulated as $S_3$-symmetries of the analytic element $(\mc I^{ijk})_{(i,j,k)\in S_3}$ 
look even more well-rounded.

\end{rem}

\subsection{Subbundle} 

Denote $\tilde W=\{(I_1,I_2,I_3)\in \Q_p^3\ |\ I_1+I_2+I_3=0\}$.
System  \eqref{KZZ} of KZ equations defines a flat connection on the trivial bundle
$ \tilde W\times \tilde{\frak D}_0 \to \tilde{\frak D}_0$. The flat sections of that bundle are solutions of
 system  \eqref{KZZ} of KZ equations.

\vsk.2>
For any $\al \in \tilde{\frak D}$ such that $\al\in \tilde{\frak D}^\si$ the vector
$\mc I^\si(\al)$ spans a one-dimensional subspace $\tilde U_\al\subset \tilde W$. That
 subspace does not depend on
$\si$ such that $\al\in \tilde{\frak D}^\si$. The union of these subspaces defines a one-dimensional 
subbundle $\tilde {\mc U}\to \tilde{\frak D}$ of the trivial bundle $ \tilde W\times \tilde{\frak D} \to \tilde{\frak D}$. 

\begin{thm}
\label{thm inv} 

The subbundle  $\tilde {\mc U}\to \tilde{\frak D}$ is invariant with respect to the KZ connection on 
$ \tilde W\times \tilde{\frak D} \to \tilde{\frak D}$.
\end{thm}

\begin{proof} For any $\si \in S_3$ 
the subbundle $\tilde {\mc U}\to \tilde{\frak D}$ is generated by the flat section 
$I^\si$ near the points where $u^\si_2=0$. 
Hence the subbundle $\tilde {\mc U}\to \tilde{\frak D}$ is generated by a flat section 
near any point of $\tilde{\frak D}$, see Section \ref{sec A18} and \cite[Lemma 3.1]{Dw}.
\end{proof}

\begin{rem}
For any $\si\in S_3$ the flat section $I^\si$ generates the subbundle  $\tilde {\mc U}\to \tilde{\frak D}$ 
 near the points where $u^\si_2=0$. The power series $I^\si$ considered over $\C$ 
is the expansion of an integral over a cycle vanishing at the points where $u_2^\si=0$.
The analytic continuation over $\C$ of that integral over that vanishing cycle could not generate a one-dimensional subbundle
of the trivial bundle $\tilde W\times \tilde\fD\to \tilde \fD$ since the monodromy representation of the complex KZ equations
in this case is irreducible. In contrast with this fact over $\C$, the $p$-adic power series solutions $I^\si$,
$\si\in S_3$, defined at different points
glue together into a single line bundle $\tilde {\mc U}\to \tilde{\frak D}$. This line bundle is what Dwork calls a
{\it $p$-adic cycle}. This $p$-adic phenomenon was stressed by Dwork in \cite{Dw} 
who titled  his paper  {\it $P$-adic Cycles}. 

\end{rem}

\begin{rem} 
The invariant subbundles of the KZ connection over $\C$ usually are related to some additional conformal block
constructions, see \cite{FSV1, FSV2, SV2, V7}. Apparently the subbundle $\tilde {\mc U}\to \tilde{\frak D}$  is of a
different $p$-adic nature, cf. \cite{V7}.

\end{rem}

\subsection{Boundedness}
Let $\si\in S_3$ and $\al \in \tilde{\frak D}^\si$. For $w\in W$ let
$I(z;w)$ be the germ at $\al$ of the solution of the 
KZ equations with initial condition $I(\al,w)=w$. 
By formula \eqref{Lhu}, the coordinates of $I(z;w)$
have the form
\bea
(u_1^\si)^{-3/2}(- h(u_2^\si) -2h'(u_2^\si)u_2^\si), \quad
(u_1^\si)^{-3/2} 2h'(u_2^\si),\quad
(u_1^\si)^{-3/2}( h(u_2^\si) +2h'(u_2^\si)(u_2^\si-1)),
\eea
where $h$ is the germ at the point $u^\si_2=u^\si_2(\al)$  of a solution of equation \eqref{HEa}.
We say that the germ $I(z;w)$ is bounded
if  each of the germs
$- h(u_2^\si) -2h'(u_2^\si)u_2^\si$,
$ 2h'(u_2^\si)$,
$h(u_2^\si) +2h'(u_2^\si)(u_2^\si-1)$
is bounded in its disc of convergence.

\begin{thm}
\label{thm bound}

The germ $I(z;w)$ is bounded if and only if $w\in \tilde U_\al$.

\end{thm}

\begin{proof} Let $w\in \tilde U_\al$. Then the germ $h$ belongs to the corresponding subspace
$U_{u^\si_2(\al)}$ defined in 
Section \ref{sec A18}. By \cite[Lemma 4.2]{Dw} 
the germ $h$ is bounded in its disc of convergence, see Section \ref{sec A19}.
Hence each of the three germs $- h(u_2^\si) -2h'(u_2^\si)u_2^\si$,
$ 2h'(u_2^\si)$,
$h(u_2^\si) +2h'(u_2^\si)(u_2^\si-1)$ is bounded in its disc of convergence. 

If $w\not\in \tilde U_\al$\,,\ then  $h\not\in U_{u^\si_2(\al)}$.  By \cite[Lemma 4.2]{Dw} 
the germ $h$ is unbounded in its disc of convergence. Then at least one
of the three germs $- h(u_2^\si) -2h'(u_2^\si)u_2^\si$,
$ 2h'(u_2^\si)$,
$h(u_2^\si) +2h'(u_2^\si)(u_2^\si-1)$ is unbounded in its disc of convergence. 
\end{proof}

\subsection{More domains}
Denote
\bean
\label{tD34}
&
\tilde {\frak D}_3 =\big\{(z_1,z_2,z_3)\in \tilde {\frak D}_0\  \big|\ \
\frac{z_2-z_3}{z_1-z_3}\in \Z_p,\ \big| \frac{z_2-z_3}{z_1-z_3}\big|_p = 1, \
\big| \frac{z_2-z_3}{z_1-z_3}-1\big|_p = 1 \big\},
\\
\notag
&
\tilde {\frak D}_4 =\big\{(z_1,z_2,z_3)\in \tilde {\frak D}_0\  \big|\ \
\frac{z_2-z_3}{z_1-z_3}\in \Z_p,\ \eps <\big| \frac{z_2-z_3}{z_1-z_3}\big|_p < 1 \big\},
\eean
where $\eps$ is the same number as in \eqref{D34}.

\subsection{Frobenius map on solutions of KZ equations}
\label{sec A9}

 Formula \eqref{Frob} describes the Frobenius map on solutions of equation 
\eqref{DDE}. Solutions of equation \eqref{DDE} are identified with solutions of 
the KZ system \eqref{KZZu} by formulas
 \eqref{H-KZ} and \eqref{KZ-H}. That allows us to define
 the Frobenius map on solutions of the KZ system \eqref{KZZu}.
 
\vsk.2>
Denote
\bean
\label{matr}
\mc B(u_1,u_2) =
u_1^{3/2}
\begin{pmatrix}
 -1 & -u_2 & 0
 \\
 0   &   1/2 & 0
\end{pmatrix} ,
\quad 
\mc C(u_1,u_2) =
u_1^{-3/2}
\begin{pmatrix}
 -1 & -2u_2
 \\
 0   &   2
 \\
1 & 2u_2-2
\end{pmatrix}.
\eean
we have
$\mc B(u_1,u_2)\mc C(u_1,u_2) 
=\begin{pmatrix}
 1 & 0
 \\
 0   &   1
\end{pmatrix}$, 
$\mc C(u_1,u_2)\mc B(u_1,u_2) 
=\begin{pmatrix}
 1 & 0 &0
 \\
 0   &   1&0
\\
-1&-1&0
\end{pmatrix}$. The second matrix defines the identity operator on the space 
$\tilde W=\{(I_1,I_2,I_3)\in \Q_p^3\ |\ I_1+I_2+I_3=0\}$.

\vsk.2> 

Recall the matrix $\mc F(u_2)$  defined in \eqref{U la}.
By formula \eqref{H-KZ} the matrix 
\bea
\tilde {\mc F}(u_1,u_2) = \mc C(u_1,u_2) \mc F(u_2)
\eea
 is a fundamental matrix of solutions of system \eqref{KZZu}.
 Recall the matrices $\mc A(\la)$, $M$ in \eqref{mc A}.
Denote
\bean
\label{Fr mat}
\tilde{\mc A}(u_1,u_2)  = \mc C(u_1,u_2)\,\mc A(u_2) \,\mc B((u_1)^p,(u_2)^p).
\eean
This is a $3\times 3$ matrix valued function, whose values preserve the subspace
$\tilde W\subset \Q_p^3$.

\begin{thm}
\label{thm on Frob}

We have
\bean
\label{F KZ}
\tilde{\mc A}(u_1,u_2) \,\tilde {\mc F}((u_1)^p,(u_2)^p)\,=\, \tilde {\mc F}(u_1,u_2) \,M.
\eean
The matrix $\tilde{\mc A}(u_1,u_2)$ extends to an analytic function on the domain
$\tilde{\frak D}_3\cup \tilde{\frak D}_4$.

\end{thm}

\begin{proof}
The theorem is a corollary of formula \eqref{Frob} and Dwork's statements listed in Section \ref{sec A110}.
\end{proof}

We call  $\tilde{\mc A}(u_1,u_2)$ the matrix of the Frobenius transformation
of solutions of system \eqref{KZZu} 
relative to the fundamental matrix $\tilde{\mc F}(u_1,u_2)$ on the domain 
$\tilde{\frak D}_3\cup \tilde{\frak D}_4$.

\vsk.2>

Recall the distinguished solution
\bean
\label{I 123}
\phantom{aaa}
I(u_1,u_2) = u_1^{-3/2}( - F(u_2) -2F'(u_2)u_2,\, 2F'(u_2),\,  F(u_2) +2F'(u_2)(u_2-1))
\eean

\vsk.4>
\noindent
of system \eqref{KZZu}  defined near the points where $u_2=0$, see \eqref{Lhu}. 
By \eqref{AF} we have
\bean
\label{tAF}
\tilde{\mc A}(u_1,u_2)\, I((u_1)^p,(u_2)^p) \,=\, (-1)^{(p-1)/2}\,I(u_1,u_2)
\eean
on $\tilde {\frak D}_4$.
Recall $\ell(u_1,u_2) = u_1^{-1/2}F(u_2)$ in \eqref{Kz}. Dividing both sides in
\eqref{I 123} by 
\linebreak
$\ell((u_1)^p, (u_2)^p)$ we can reformulate \eqref{tAF} as

\bean
\label{FREK}
\tilde {\mc A}(u_1,u_2) \,\mc I((u_1)^p,(u_2)^p)\, 
 =\,(-1)^{(p-1)/2} u_1^{(p-1)/2} f(u_2) \,\mc I(u_1,u_2)
 \eean

\vsk.4>
\noindent
 on $\tilde\fD_4$, see $\mc I(u_1,u_2)$ in \eqref{K/K} and $f(u_2)$ in \eqref{def etf}. 
 As in Section \ref{sec A110} we conclude with Dwork that
 relation \eqref{FREK}
 can be analytically continued to the domain $\tilde\fD_1^{(1,2,3)}\cap \tilde \fD_3$.

 \vsk.4>
 
 Equation \eqref{FREK} implies that for  any $\al \in \F^\times_p-\{1\}$, $\beta \in \F^\times_p$
such that $\om(\al) \in \fD_1$, the vector $\mc I(\om(\beta),\om(\al))$ is an eigenvector of the Frobenius matrix
$\tilde{\mc A}(\om(\beta),\om(\al))$ with eigenvalue  $\om(\beta^{(p-1)/2}) (-1)^{(p-1)/2}  f(\om(\al))$,
\bean
\label{eigK}
\phantom{aaaaaa}
\tilde{\mc A}(\om(\beta),\om(\al))\,\mc I(\om(\beta),\om(\al))\,=\,   
\om(\beta^{(p-1)/2})(-1)^{(p-1)/2}  f(\om(\al))\,\mc I(\om(\beta),\om(\al))\,.
 \eean

\vsk.4>
In this Section \ref{sec A9} we described the matrix $\tilde {\mc A}(u_1,u_2)$
of the Frobenius transformation
of solutions of system \eqref{KZZ} written in coordinates $u_1,u_2,u_3$ corresponding to the chart
labeled by the identify permutation $(1,2,3)\in S_3$,\, 
In the same way we may start with the chart corresponding
to any permutation $\si\in S_3$ and describe 
the matrix of the Frobenius transformation
of solutions of system \eqref{KZZ} written in coordinates $u_1^\si,u_2^\si,u_3^\si$.

\subsection{Eigenvalue  $\om(\beta^{(p-1)/2}) (-1)^{(p-1)/2}  f(\om(\al))$}
\label{sec A10}

\begin{thm}
\label{lem uro}

The number $\om(\beta^{(p-1)/2}) (-1)^{(p-1)/2}  f(\om(\al))$
is  the unit root of the elliptic curve $E(\al,\beta)$
defined over $\F_p$ by the affine equation
\bean
\label{al be}
w^2= \beta\, v (v-1)(v-\al).
\eean
\end{thm} 

\begin{proof}  

Assume that $\beta\in \F_p^\times$ is a square, $\beta=\ga^2$ for some $\ga\in\F_p$. 
Then on the one hand  the change of the variable $\tilde w= w/\ga$ makes $E(\al,\beta)$ isomorphic to $E(\al,1)$.
On the other hand $\beta^{(p-1)/2} = 1$ and 
$\om(\beta^{(p-1)/2}) (-1)^{(p-1)/2}  f(\om(\al)) =  (-1)^{(p-1)/2}  f(\om(\al))$, where
the last number is the unit root  of the elliptic curve $E(\al,1)$ by \cite{Dw}.

Assume that $\beta\in \F_p^\times$ is not a square.  Denote by $N_{1,\beta}$ the number of points on
$E(\al,\beta)$. Then
\bean
\label{N+N}
N_{1,1}+N_{1,\beta} = 4+4+2(p-3) = 2p+2.
\eean
Indeed the number $4+4$ corresponds to the points  $(0,0)$, $(0,1)$, $(0,\al)$, $\infty$ on 
$E(\al,\beta)$ and on $E(\al,1)$. The number $2(p-3)$ corresponds to  $p-3$
elements of $ \F_p-\{0,1,\al\}$. Namely if 
 $v_0\in \F_p-\{0,1,\al\}$,\, then exactly one of the two elements
$\beta v_0 (v_0-1)(v_0-\al), \, v_0 (v_0-1)(v_0-\al)$ is a square in $\F_p$
and exactly one of the two elliptic curves
has two points over $v=v_0$, while the other curve does not have points over $v_0$. 

It is known that  the zeta function of the curve $E(\al,\beta)$ has the form
\bean
\label{def z}
\exp\Big(\sum_{s=1}^\infty\frac{N_{s,\beta}}s T^s\Big) =
\frac{(1- R_\beta T)(1- (p/R_\beta) T)}{(1-T)(1-pT)}.
\eean
Here $N_{s,\beta}$ is the number points on  $E(\al,\beta)$ considered over the field $\F_{p^s}$, while
the number $R_\beta$ has $|R_\beta|_p=1$ and is called the unit root, for example see \cite{Mo}.
Equation \eqref{def z} implies that for any $s$ we have $N_{s,\beta}=1+p^s - R_\beta^s - (p/R_\beta)^s$.
In particular for $s=1$ we have
\bean
\label{N_1}
N_{1,\beta} = 1+p - R_\beta - p/R_\beta.
\eean
From \eqref{N+N} and \eqref{N_1} we obtain
\bea
0=R_\beta + p/R_\beta + R_1 + p/R_1 = (R_\beta+R_1) (1+p/R_\beta R_1).
\eea
Since the second factor is nonzero we conclude that
$R_\beta = - R_1$. By \cite{Dw} we have $R_1= (-1)^{(p-1)/2}  f(\om(\al))$.
Hence $R_\beta = -\, (-1)^{(p-1)/2}  f(\om(\al)) =
\om(\beta^{(p-1)/2}) (-1)^{(p-1)/2}  f(\om(\al))$. The theorem is proved.
\end{proof}

The relation between  the eigenvector $\mc I(\om(\beta),\om(\al))$ and 
the elliptic curve $E(\al,\beta)$, indicated in Theorem \ref{lem uro},
 can be explained as follows.
Over $\C$ the vector $\mc I$ is  given by integrals over cycles on
elliptic curves with equation $y^2=(x-z_1)(x-z_2)(x-z_3)$. After the change of variables
\bea
x=(z_1-z_3) w+z_3,
\quad u_1=z_1-z_3,
\quad 
u_2=\frac{z_2-z_3}{z_1-z_3},
\quad
y=(z_1-z_3)v,
\eea
the equation takes the form
\bean
\label{vw eq}
v^2 = u_1 (w-1) (w-u_2) w.
\eean
The eigenvector $\mc I(\om(\beta),\om(\al))$ corresponds to the curve in \eqref{vw eq}
with $(u_1,u_2)$ $ =$
 $ (\om(\beta),\om(\al))$, and $ (\om(\beta),\om(\al))\equiv (\beta, \al)$ mod $p$.

\vsk.2>
It is more surprising that system \eqref{KZZ} of KZ equations
gives a bit more arithmetic information
than the hypergeometric equation \eqref{DDE},
 despite the fact system \eqref{KZZ} and equation \eqref{DDE} are equivalent by
 \eqref{H-KZ} and \eqref{KZ-H}.
Indeed Dwork's eigenvectors in \eqref{FRE} give unit roots of elliptic curves $E(\al,1)$ while
the eigenvectors in \eqref{eigK} coming from the KZ equations give unit roots of more general
elliptic curves $E(\al,\beta)$.

\subsection{Approximation of analytic element  $(\mc I^{ijk})_{(i,j,k)\in S_3}$ by rational functions}

Let $s$ be a positive integer.

\subsubsection{}

Let $M = (p^s-1)/2$,\  \ $\Phi_{p^s}(x,z) 
=
\prod_{i=1}^3(x-z_i)^{M}$,
\bea
P_{p^s}(x,z) 
=
\Big(\frac {\Phi_{p^s}(x,z)}{x-z_1}, \frac {\Phi_{p^s}(x,z)}{x-z_2} ,\frac {\Phi_{p^s}(x,z)}{x-z_3}\Big)
=\sum_i P^{i}_{p^s}(z) \,x^i\,.
\eea
Denote $I^{[p^s-1]}_{p^s}(z)\,=\, P^{p^s-1}_{p^s}(z)$. 
The functions
\bea
I^{[p^{s}-1]}_{p^{s}}(z), \quad p I^{[p^{s-1}-1]}_{p^{s-1}}(z), \ \ \dots ,\  \ 
 p^{s-2} I^{[p^2-1]}_{p^2}(z),
\quad 
p^{s-1} I^{[p-1]}_{p}(z),
\eea 
are solutions of system \eqref{KZZ} modulo $p^s$ by Theorem \ref{thm Fp}.

\vsk.2>
Let  $\Phi_{p^s}(x,z) = \sum_{i} \Phi_{p^s}^i(z)x^i$. 
Denote
\bean
\ell_{p^s}(z) = \Phi_{p^s}^{p^s-1}(z).
\eean

\subsubsection{} For $k=1,2,3$ let
\bea
 P_{k,p^s}(v,z):=
P_{p^s}(v+z_k,z) 
=\sum_i P^{i}_{k, p^s}(z) \,v^i\,.
\eea
Denote $I^{[p^s-1]}_{k,p^s}(z)\,=\, P^{p^s-1}_{k,p^s}(z)$.
The functions
\bea
I^{[p^{s}-1]}_{k,p^{s}}(z), \quad p I^{[p^{s-1}-1]}_{k,p^{s-1}}(z), \ \ \dots ,\  \ 
 p^{s-2} I^{[p^2-1]}_{k,p^2}(z),
\quad 
p^{s-1} I^{[p-1]}_{k,p}(z),
\eea 
are solutions of system \eqref{KZZ} modulo $p^s$ by Theorem \ref{thm tp^s}.

\vsk.2>

\vsk.2>
Let  $\Phi_{k, p^s}(v,z) := \Phi_{p^s}(v+z_k, z) = \sum_{i} \Phi_{k, p^s}^i(z)v^i$. 
Denote
\bean
\ell_{k, p^s}^{[p^{s}-1]}(z) = \Phi_{k, p^s}^{p^s-1}(z).
\eean

\subsubsection{}
Recall the  homomorphisms 
$\Z \to \Zs$, $\Z[z] \to (\Zs)[z]$,  $\Z[z]^3 \to (\Zs)[z]^3$
denoted by $\pi_s$.
Recall the subring  $\Z[z]_{p^r}\subset \Z[z]$ of 
quasi-constants modulo $p^r$. 

\subsubsection{}

Define the filtration
\bean
\label{flt}
0 = \mathcal{M}_{p^s}^0
\subset \mathcal{M}_{p^s}^1\subset \dots\subset \mathcal{M}_{p^s}^{s-1} \subset \mathcal{M}_{p^s}^s = \mathcal{M}_{p^s}\,,
\eean

\noindent
where
$\mathcal{M}_{p^s}^t\,=\,\big\{ \pi_{s}\big(\sum_{r=1}^t c_{r}(z) \,p^{s-r} I^{[p^r-1]}_{p^r}(z) \big)
\ |\ c_{r}(z)\in\Z[z]_{p^{r}}\big\},$
$ t=1,\dots,s$.\
 Every element of $\mathcal{M}_{p^s}$ is a polynomial
solution of system \eqref{KZZ} with coefficients in $\Zs$.
\vsk.2>

Define the filtration
\bean
\label{flT}
0 = \mathcal{M}_{k,p^s}^0
\subset \mathcal{M}_{k,p^s}^1\subset \dots\subset \mathcal{M}_{k,p^s}^{s-1} 
\subset \mathcal{M}_{k,p^s}^s = \mathcal{M}_{k,p^s}\,,
\eean

\noindent
where
$\mathcal{M}_{k,p^s}^t\,=\,\big\{ \pi_{s}\big(\sum_{r=1}^t c_{r}(z) \,p^{s-r} I^{[p^r-1]}_{k,p^r}(z) \big)
\ |\ c_{r}(z)\in\Z[z]_{p^{r}}\big\},$
$ t=1,\dots,s$.\
 Every element of $\mathcal{M}_{p^s}$ is a polynomial
solution of system \eqref{KZZ} with coefficients in $\Zs$ by Theorem \ref{thm tp^s}.
\vsk.2>

\vsk.2>
By Theorem \ref{thm tp^s} filtrations \eqref{flt} and \eqref{flT} coincide,  
$\mathcal{M}_{p^s}^t=\mathcal{M}_{k,p^s}^t$ for any $k,t$.

\subsubsection{}

Define the filtration
\bean
\label{Lfl}
0 = \mathcal{L}_{p^s}^0
\subset \mathcal{L}_{p^s}^1\subset \dots\subset \mathcal{L}_{p^s}^{s-1} 
\subset \mathcal{L}_{p^s}^s = \mathcal{L}_{p^s}\,,
\eean

\noindent
where
$\mathcal{L}_{p^s}^t\,=\,\big\{ \pi_{s}\big(\sum_{r=1}^t c_{r}(z) \,p^{s-r} \ell^{[p^r-1]}_{p^r}(z) \big)
\ |\ c_{r}(z)\in\Z[z]_{p^{r}}\big\},$
$ t=1,\dots,s$.\

\vsk.2>

Define the filtration
\bean
\label{Lfk}
0 = \mathcal{L}_{k,p^s}^0
\subset \mathcal{L}_{k,p^s}^1\subset \dots\subset \mathcal{L}_{k, p^s}^{s-1} 
\subset \mathcal{L}_{k, p^s}^s = \mathcal{L}_{k,p^s}\,,
\eean

\noindent
where
$\mathcal{L}_{k,p^s}^t\,=\,\big\{ \pi_{s}\big(\sum_{r=1}^t c_{r}(z) \,p^{s-r} \ell^{p^r-1]}_{k,p^r}(z) \big)
\ |\ c_{r}(z)\in\Z[z]_{p^{r}}\big\},$
$ t=1,\dots,s$.\

\vsk.2>
It is easy to see that filtrations \eqref{Lfl} and \eqref{Lfk} coincide,  
$\mathcal{L}_{p^s}^t=\mathcal{L}_{k,p^s}^t$ for any $k,t$.

\subsubsection{}
\label{sec A106}

Let $u=(u_1,u_2,u_3)$ be the coordinates in the chart corresponding to $(1,2,3)\in S_3$, see \eqref{id co}.
Consider the functions $I^{[p^s-1]}_{3,p^s}(z)\in \mc M_{p^s}$, \  
$\ell^{[p^s-1]}_{3,p^s}(z)\in \mc L_{p^s}$.   Denote
\bean
\label{Iell}
\hat I^{[p^s-1]}(u):= I^{[p^s-1]}_{3,p^s}(z(u)), 
\qquad
\hat \ell^{[p^s-1]}(u) :=\ell^{[p^s-1]}_{3,p^s}(z(u)).
\eean
We have 
\bean
\label{gc}
&\phantom{aaa}
(-1)^{\frac{p^s-3}2} \hat I^{[p^s-1]} =\,
u_1^{\frac{p^s-3}2} 
 \sum_{a=0}^{\frac{p^s-1}2}\Big(\binom{\frac{p^s-3}2}{a}\binom{\frac{p^s-1}2}{a},
\binom{\frac{p^s-1}2}{a+1}\binom{\frac{p^s-3}2}{a},
\binom{\frac{p^s-1}2}{a+1}\binom{\frac{p^s-1}2}{a}\Big) \,u_2^a,
\\
\label{Lgc}
&
(-1)^{\frac{p^s-1}2} \hat \ell^{[p^s-1]}
=
u_1^{\frac{p^s-1}2}
 \sum_{a=0}^{\frac{p^s-1}2}\binom{\frac{p^s-1}2}{a}^2 u_2^a\,,
\eean
see formula \eqref{gc} and \eqref{g=1 sc}. 

\vsk.2>
Notice that $ \sum_{a=0}^{\frac{p^s-1}2}\binom{\frac{p^s-1}2}{a}^2 u_2^a$ is a solution of eqution
\eqref{HEa} modulo $p^s$.

\vsk.2>

As $s\to\infty$ the sequence $\big((-1)^{\frac{p^s-3}2} \hat I^{[p^s-1]}\big)_{s=1}^\infty$ of vector-valued
polynomials uniformly converges  to the series
$I$ in \eqref{hIu}  near the points where $u_2=0$, see Theorem \ref{thm 10.3}.
Similarly as $s\to\infty$ the sequence $\big((-1)^{\frac{p^s-3}2} \hat \ell^{[p^s-1]}\big)_{s=1}^\infty$ 
of scalar polynomials  uniformly 
converges  to the series $\ell$ in \eqref{Kz} near the points where $u_2=0$.

\begin{cor}
\label{cor appr}
As $s\to\infty$ the sequence $\big(- \hat I^{[p^s-1]}\big/\hat \ell^{[p^s-1]}\big)_{s=1}^\infty$
of vector-valued rational functions
uniformly converges  to the function
$\mc I$ in \eqref{K/K} near the points where $u_2=0$.

\end{cor}

\subsubsection{}

Let $\si=(i,j,k)\in S_3$.
Similarly to Section \ref{sec A106} consider the coordinates $u^\si$
 and the functions
 $I^{[p^s-1]}_{k,p^s}(z)\in \mc M_{p^s}$, \  
$\ell^{[p^s-1]}_{k,p^s}(z)\in \mc L_{p^s}$.   Denote
\bean
\label{Iell}
\hat I^{[\si, p^s-1]}(u^\si):= I^{[p^s-1]}_{k,p^s}(z(u^\si)), 
\qquad
\hat \ell^{[\si, p^s-1]}(u^\si) :=I^{[p^s-1]}_{k,p^s}(z(u^\si)).
\eean
Similarly to Section \ref{sec A106} we obtain the following corollary.

\begin{cor}
\label{cor app si}
As $s\to\infty$ the sequence $\big(- \hat I^{[\si, p^s-1]}\big/\hat \ell^{[\si,p^s-1]}\big)_{s=1}^\infty$
of vector-valued rational functions
uniformly converges  to the function
$\mc I^\si$ in \eqref{K/K} near the points where $u_2^\si=0$.

\end{cor}

\subsubsection{}

This Appendix \ref{appendix} is devoted to the relation between the analytic element 
\linebreak
$(\mc I^{ijk})_{(i,j,k)\in S_3}$ and Dwork's theory in \cite{Dw}.

\vsk.2>

As additional information, for any $\si=(i,j,k)\in S_3$,
 Corollary \ref{cor app si} indicates the
sequence of polynomials 
$I^{[p^s-1]}_{k,p^s}(z)\in \mc M_{p^s}$, \  
$\ell^{[p^s-1]}_{k,p^s}(z)\in \mc L_{p^s}$ whose ratio $p$-adically
tends to the function $\mc I^{ijk}$ near the points $u_2^\si=0$
 where the function 
$\mc I^{ijk}$ is initially defined.

\subsection{Further directions}
\label{sec A12}

In Sections \ref{sec DE} - \ref{sec10}  we considered 
system \eqref{KZ} of KZ equations with parameter $n=2g+1$ and constructed 
polynomial solutions of system \eqref{KZ} modulo $p^s$. We defined the
module $\mc M_{p^s}$ of the constructed solutions and studied the limit
of $\mc M_{p^s}$ as $s\to\infty$. Namely we considered a special coordinate system
$u=u(z)$ in \eqref{uz} associated with one of the asymptotic zones of the KZ equations
and showed that in this coordinate system the limit of
$\mc M_{p^s}$ as $s\to\infty$ produces a $g$-dimensional space of solutions of system
\eqref{KZ}  over $p$-adic numbers $\Q_p$ in the neighborhood of the point $u=0$.

\vsk.2>
Constructions   in this appendix  for $g=1$
and Dwork's theory in \cite{Dw} suggest  the following  project.
Consider all asymptotic zones of system \eqref{KZ},   see their definition for example  in \cite{V2}.
The asymptotic zones are labeled by suitable trees $T$. These trees are analogs of the elements $\si\in S_3$ in the appendix.
 Each asymptotic zone has a distinguished system of coordinates
$u^T$. Probably, for every asymptotic zone the limit of 
$\mc M_{p^s}$ as $s\to\infty$ produces a $g$-dimensional  space $V_T$ of solutions of system
\eqref{KZ} considered over $\Q_p$ in a neighborhood of the point $u^T=0$.
Probably the spaces $V_T$ of $p$-adic solutions, defined at different places $u^T=0$, analytically continue
into a single global invariant $g$-dimensional vector subbundle of the associated KZ connection 
on the trivial vector  bundle of rank $2g$. 
Following Dwork and Theorem \ref{thm bound} we may expect that this subbundle is spanned
at any point of the base by the germs of all solutions of the KZ equations bounded in their polydiscs of convergence.
This subbundle 
would give a generalization of the line subbundle generated by the analytic
element $(\mc I^{ijk})_{(i,j,k)\in S_3}$ constructed in this appendix.  Probably, that 
$g$-dimensional subbundle 
will determine
the set of unit roots of the curves with equation $y^2= \prod_{i=1}^n(x-z_i)$ over the field $\F_p$
similarly to how it is done in Sections \ref{sec A9} and \ref{sec A10} for the elliptic curves.

\end{document}